 \documentclass[reqno]{amsart}
\usepackage{amsmath,amstext, amsthm,amsfonts,amssymb,amsthm} 
\usepackage{epsfig, cite}
\newtheorem{thm}{Theorem}[section]
\newtheorem{cor}[thm]{Corollary}
\newtheorem{lem}[thm]{Lemma}
\newtheorem{prop}[thm]{Proposition}
\theoremstyle{definition}

\theoremstyle{remark}
\newtheorem{rem}[thm]{Remark}
\numberwithin{equation}{section}

\newcommand{\abs}[1]{\left\vert#1\right\vert}
\newcommand{\set}[1]{\left\{#1\right\}}

\newcommand{\Z}{\mathbb Z}
\newcommand{\C}{\mathbb C}
\newcommand{\R}{\mathbb R}

\newcommand{\be}{\begin{equation}}
\newcommand{\ee}{\end{equation}}
\newcommand{\bea}{\begin{eqnarray}}
\newcommand{\eea}{\end{eqnarray}}

\newcommand{\ba}{\begin{array}}
\newcommand{\ea}{\end{array}}

\begin{document}

\title[]{Functional definitions for $q-$analogues of Eulerian functions and applications}%
\author{ Ahmad El-Guindy$^\dag$ and Zeinab Mansour$^\ddag$ }
\address{$\dag$ Permanent Address: Department of Mathematics, Faculty of Science, Cairo University, Giza, Egypt 12613\newline
$\dag$ Current Address: Science Program, Texas A\&M University at Qatar, Doha, Qatar 23874 \newline 
$\dag\dag$ Department of Mathematics, Faculty of Science, King Saud University, Riyadh,
P. O. Box 2455, Riyadh 11451, Kingdom of Saudi Arabia}
\email {a.elguindy@gmail.com, zeinabs98@hotmail.com}%

\thanks{The research of Zeinab Mansour is supported by NPST Program of King Saud University; project number 10-MAT1293-02.}

\begin{abstract}
We explore a number of functional properties of the $q$-gamma function and a class of its quotients; including the $q$-beta function. We obtain formulas for all higher logarithmic derivatives of these quotients and give precise conditions on their sign. We prove how these and other functional properties,  such as the multiplication formula or the asymptotic expansion, together with the fundamental functional equation of the $q$-gamma function uniquely define those functions. We also study  reciprocal ``relatives" of the fundamental $q$-gamma functional equation, and prove uniqueness of solution results for them.  In addition, we also use a reflection formula of  Askey to derive expressions relating the classical sine function and the number $\pi$ to the $q$-gamma function. Throughout we highlight the similarities and differences between  the cases $0<q<1$ and $q>1$.
\end{abstract}
\maketitle
\section{{\bf Introduction and Preliminaries}}

Let $q$ be a positive number, $0<q<1$. For  $n\in\mathbb{N}=\set{0,1,\ldots}$, and  $a\in\mathbb{C}$, the $q$-shifted
 factorial is  defined by
 \begin{equation}\label{np:2}
(a;q)_0:=1,\;\;(a;q)_n:=\prod_{k=0}^{n-1}(1-a q^k).\;\;
 \end{equation}
   The limit, $\displaystyle\lim_{n\to
\infty}(a;q)_n$, is   denoted by $(a;q)_\infty$.

 Jackson~\cite{GR,Jac11} introduced a $q$-analogue of the gamma function by

\begin{equation}\label{qGammaDefinition}
\Gamma_q(x):=\frac{(q;q)_{\infty}}{(q^x;q)_{\infty}}(1-q)^{1-x},
\quad x\in\mathbb{C}-\set{0,-1,-2,\ldots},\;q\in(0,1).
\end{equation}
For $q>1$, the $q$-gamma function is defined by
\be
\Gamma_q(x)=\dfrac{(q^{-1};q^{-1})_{\infty}}{(q^{-x};q^{-1})_{\infty}}(q-1)^{1-x}q^{\frac{x(x-1)}{2}},\;x\in\mathbb{C}-\set{0,-1,-2,\ldots}.
\ee
It is straightforward to conclude that for $q>0$\;

\be\label{gammaq-1} 
\Gamma_q(x)=q^{\frac{(x-1)(x-2)}{2}}\Gamma_{q^{-1}}(x), \quad\mbox{for all }\;x\in\mathbb{C}-\set{0,-1,-2,\ldots}.
\ee
Also, if for $q>0$ we write $[x]_q:=\dfrac{1-q^x}{1-q}$,   then the $q$-gamma function satisfies the following fundamental functional equation
\;

\begin{equation}\label{char2}
\begin{split}
\Gamma_q(x+1)&=[x]_q\Gamma_q(x),
\Gamma_q(1)=1.
\end{split}
\end{equation}
Furthermore, for $x \in \C-\set{0,-1,-2,\ldots}$ we have

\be\label{limit}
\lim_{q \rightarrow 1} \Gamma_q(x)=\Gamma (x),
\ee
and
\begin{equation}\label{limit0}
\lim_{q \rightarrow 0^+} \Gamma_q(x)=1.
\end{equation}
Thus, it is natural to set $\Gamma_0(x):=1$  and $\Gamma_1(x):=\Gamma(x)$. In additon, $\Gamma_q(x)$ is continuous if viewed as a function in the two variables $x$ and $q$ for $x>0$ and $q\geq 0$; in other words
\bea\label{gammalim}
\begin{split}
\lim_{(y,p)\to(x,q)} \Gamma_p(y)&=\Gamma_q(x), \ \ \ (q>0),\\
\lim_{(y,p)\to (x,0^+)}\Gamma_p(y)&=1=\Gamma_0(x).
\end{split}
\eea
(See \cite{aar} for these and other facts about the gamma and $q$-gamma functions.)

Properties \eqref{char2}, and \eqref{limit} show that  the $q$-gamma function is indeed \emph{an} analogue of the gamma function, but it is not apriori clear that it is the most natural analogue. This is in fact similar to the situation with the gamma function itself, where one easily sees that it interpolates the defining properties of the factorial to the complex plane, but the fact that it is the ``most natural" extension follows from realizing that it is the unique function satisfying certain additional properties. For instance  the Bohr-Mollerup Theorem asserts that the gamma function is the unique interpolation of the factorial whose logarithm is a convex function.  Askey \cite{Askey78} proved the following $q$-analogue of the Bohr-Mollerup theorem.

\begin{thm}{Askey \cite[Theorem 3.1]{Askey78}}

The function $\Gamma_q(x)$, $0<q<1$,  is the unique logarithmically convex function that satisfies the functional equation,
\[f(x+1)=[x]_q f(x), \quad f(1)=1,\;x>0.\]
\end{thm}
Moak~\cite{Moak} proved slightly different results for the case $q>1$.

\begin{thm}{Moak \cite[Theorem 1 and Theorem 2]{Moak} }

If $q>1$, and $f:(0,\infty)\to(0,\infty)$ is a positive function that satisfies
\be
\begin{split}
&f(x+1)=[x]_q f(x),    
f(1)=1,
\end{split}
\ee
and either of the following conditions
\be
\frac{d^3}{dx^3}\log f(x)\leq 0,
\ee
\be\label{Moakinequality}
\frac{d^2}{dx^2}\log f(x)\geq \log\/q,
\ee
then $f(x)=\Gamma_q(x)$.
\end{thm}
It is worth noting that the logarithmic convexity of the $q$-gamma function is proved by Askey~\cite{Askey78} for $0<q<1$ and by Moak~\cite{Moak} for $q>1$. Furthermore, Kairies and Muldoon \cite{K-M} also provide characterizations of the $q$-gamma function by means of $\Gamma_q(1)=1$ and \eqref{char2} together with
properties of monotonicity (or ultimate monotonicity or complete monotonicity). Those characterizations are also generalizations of similar properties of the classical gamma function.

One of our aims in this work is to extend some of the results in \cite{Askey78, Moak, K-M} to more general functional equations, where we show that, under certain conditions, their unique solutions are quotients of $q$-gamma functions. We also obtain $q$-analogues  of characterizations of the  classical gamma and beta functions of Anastassiadis \cite{anasta}. In addition, we derive formulas relating $\pi$ and $\sin(\pi x)$ to certain expressions in the $q$-gamma function. Since we are interested in investigating characterization propertied for  the $q$-gamma and $q$-beta functions on the positive real line,  from now on, unless otherwise stated,  we assume that $x$ is a positive real number.


\section{\bf The $q$-Digamma Function and Logarithmic Convexity}
The $q$-digamma function $\psi_q$, $q>0$,  is defined by

\begin{equation}
\psi_q(x)=\frac{d}{dx}\log \Gamma_q(x)=\frac{\Gamma'_q(x)}{\Gamma_q(x)}.
\end{equation}

By direct calculation, one can see that if $q\in(0,1)$, then
\begin{equation}\label{psi defn}
\begin{split}
\psi_q(x)&=-\log(1-q)+\log\/q\sum_{j=0}^{\infty}\dfrac{q^{x+j}}{1-q^{x+j}}\\
&=-\log(1-q)+\log\/q\sum_{j=1}^{\infty}\dfrac{q^{xj}}{1-q^j},
\end{split}
\end{equation}
and
\begin{equation}\label{gamma log-convex}
\psi_q'(x)=\dfrac{d^2}{dx^2}\log\Gamma_q(x)=\log^2\,q\sum_{n=0}^{\infty}\dfrac{q^{x+n}}{(1-q^{x+n})^2}>0.
\end{equation}

If $q>1$, then using \eqref{gammaq-1} we get
\be
\psi_q(x)=\psi_{q^{-1}}(x)+\left(x-\frac 32\right)\log\/q.
\ee

Recall that a function $f$ is called \emph{completely monotone} on a set $A$ if for all integers $n\geq 1$ and all $x\in A$ we have
\be\label{complete}
(-1)^n\frac{d^n}{dx^n} f(x) \geq 0. 
\ee
It was proved proved in \cite[Theorem 2.2]{Ism-Lor-Muldoon} that $-\frac{d}{dx} \log((1-q)^x \Gamma_q(x))$ is completely monotone, and in \cite{GrinshIsm} this was extended to the logarithm of certain quotients of $q$-gamma functions under certain conditions; among which is  for the number of $q$-gamma factors in the numerator and denominator to be the same. In the following series of results, we generalize the results from \cite{Ism-Lor-Muldoon} and \cite{GrinshIsm} by obtaining a full description of the sign of the second and all higher derivatives of logarithms of arbitrary quotients of $q$-gamma functions. Specifically for $q>0$ and $a_1,\dots, a_r, b_1,\dots, b_s\geq 0$ we set
\be\label{frs}
f(x):=f(x;q):=\dfrac{\prod_{i=1}^{r}\Gamma_q(x+a_i)}{\prod_{j=1}^{s}\Gamma_q(x+b_j)}.
\ee
We also consider
\be\label{grs}
g(x):=g(x;q)=(1-q)^{(r-s)x}f(x;q), 
\ee
Note that $g=f$ for $r=s$. In addition, we associate to the parameters $a_i$ and $b_j$ the quantity
\be\label{upsilon}
\Upsilon:=\Upsilon_q(a_1,\dots,a_r;b_1,\dots,b_s):=\sum_{i=1}^r q^{a_i}-\sum_{j=1}^s q^{b_j}.
\ee
We start by deriving a condition for the  eventual monotonicity of these functions for $q\in (0,1)$.

\begin{prop}\label{quotientdecrease}
Let $f(x)$ and $g(x)$ be as in \eqref{frs} and \eqref{grs}, respectively, and assume $0<q<1$. Then there exists $M\geq 0$ such that $f(x;q)$ and $g(x;q)$ are monotone for all $x>M$. More specifically, we have the following cases.
\begin{enumerate}
\item If $r>s$ then $f$ is increasing for $x>M\geq 0$.

\medskip
\item If $r<s$ then $f$ is decreasing for $x>M\geq 0$.

\medskip
\item If $\Upsilon<0$ (resp. $\Upsilon>0$), then there exists $M\geq 0$ such that $g(x)$ is increasing (resp. decreasing) for $x>M$. 
\end{enumerate} 

\end{prop}

\begin{proof}
Since $f(x)$ is always positive, we see that $f'(x)>0$  if and only if $\frac{d}{dx}\log f(x)>0$, and the same is true for $g(x)$.
We have
\be\label{logf'}
\begin{split}
\frac{d}{dx}\log f(x)&=\sum_{i=1}^r \psi_q(x+a_i)-\sum_{j=1}^s \psi_q(x+b_j)\\
&=(s-r)\log(1-q)+q^x\log q\sum_{l=0}^\infty\left(\sum_{i=1}^r \frac{q^{a_i+l}}{1-q^{x+a_i+l}}-\sum_{j=1}^s \frac{q^{b_j+l}}{1-q^{x+b_j+l}}\right),
\end{split}
\ee
and
\be\label{logg'}
\frac{d}{dx}\log g(x)=q^x\log q\sum_{l=0}^\infty\left(\sum_{i=1}^r \frac{q^{a_i+l}}{1-q^{x+a_i+l}}-\sum_{j=1}^s \frac{q^{b_j+l}}{1-q^{x+b_j+l}}\right).
\ee
Since for any $y>0$ and $l\geq 0$ we have 
\[
1<\frac{1}{1-q^{x+y+l}}\leq \frac{1}{1-q^y},
\]
and from \eqref{logf'} we get
\be
\begin{split}
&(s-r)\log(1-q)+q^x\log q\left(\sum_{i=1}^r\frac{q^{a_i}}{(1-q)}-\sum_{j=1}^s \frac{q^{b_j}}{(1-q^{b_j})(1-q)}\right)\leq \\
&\frac{f'(x)}{f(x)}\leq (s-r)\log(1-q)+q^x\log q\left(\sum_{i=1}^r\frac{q^{a_i}}{(1-q^{a_i})(1-q)}-\sum_{j=1}^s \frac{q^{b_j}}{1-q}\right).
\end{split}
\ee
Set 
\[L:=|\log q|\cdot\max\left(\left|\sum_{i=1}^r\frac{q^{a_i}}{(1-q^{a_i})(1-q)}-\sum_{j=1}^s \frac{q^{b_j}}{1-q}\right|, \left|\sum_{i=1}^r\frac{q^{a_i}}{(1-q)}-\sum_{j=1}^s \frac{q^{b_j}}{(1-q^{b_j})(1-q)}\right|\right).\]
 Since $\lim_{x\to\infty}q^x=0$  we that if $(s-r)\log(1-q)\neq 0$, we can choose $M$ large enough so that $x>M$ implies
\[
q^xL<\frac{|(s-r)\log(1-q)|}{2},
\] 
and it follows that,  for such $x$,  $f'(x)$  must have the same sign as $(s-r)\log(1-q)$; this proves parts (1) and (2) of the theorem. 

To prove part (3) we use that for $x,y>0$ and $l\geq 0$ we have
\[
1<\frac{1}{1-q^{x+y+l}}\leq \frac{1}{1-q^{x+y}}
\]
to get 
\be
\begin{split}
&q^x\log q\left(\sum_{i=1}^r\frac{q^{a_i}}{(1-q)}-\sum_{j=1}^s \frac{q^{b_j}}{(1-q^{x+b_j})(1-q)}\right)\leq \\
&\frac{g'(x)}{g(x)}\leq q^x\log q\left(\sum_{i=1}^r\frac{q^{a_i}}{(1-q^{x+a_i})(1-q)}-\sum_{j=1}^s \frac{q^{b_j}}{1-q}\right),
\end{split}
\ee
\
from which we get
\be
\frac{q^x\log q}{(1-q)}\frac{\Upsilon+h_1(q^x)}{\prod_{j=1}^s(1-q^{x+b_j})}\leq \frac{g'(x)}{g(x)}\leq\frac{q^x\log q}{(1-q)}\frac{\Upsilon+h_2(q^x)}{\prod_{i=1}^r(1-q^{x+a_i})},
\ee

where $h_1(q^x)$ and $h_2(q^x)$ are polynomials with no constant term. Since the limit of such a polynomial as $x\to \infty$ is $0$, we can choose $M$ large enough so that $\max(|h_1(q^x)|, |h_2(q^x)|)<\frac{|\Upsilon|}{2}$ for $x>M$. It follows that for such $x$, $g'(x)$ has the same sign as $\Upsilon\log q$, and part (3) follows, completing the proof of the theorem. 
\end{proof}

Next we turn our attention to the sign of higher derivatives of $f(x;q)$. We start by obtaining  a relatively simple expression for higher derivatives of the $q$-digamma function.

\begin{lem}
Let $P_n(x)$ be the sequence of Eulerian polynomials given by the recursion

\be\label{eulerpoly}
\begin{split}
&P_0(x)=1,\\
&P_{n+1}(x)=(nx+1)P_n(x)+x(1-x)P_n'(x).
\end{split}
\ee
For any  $n \geq 1$ we have 
\be\label{basicder}
\frac{d^n}{dx^n}\log [x]_q=-\log^n (q)\, \frac{q^{x}P_{n-1}(q^{x})}{(1-q^x)^n}.
\ee
Hence for $0<q<1$ and $a\geq 0$ we have
\be\label{psider}
\frac{d^n}{dx^n}\psi_q(x+a)=\log^{n+1}(q) \sum_{i=0}^\infty \frac{q^{x+a+i}P_{n}(q^{x+a+i})}{(1-q^{x+a+i})^{n+1}}.
\ee
\end{lem}

\begin{proof}
For $n\geq 0$, let $P_{n}(x)$ be the sequence of  \emph{functions} $P_{n}$ satisfying \eqref{basicder}. Since 
\be\label{psi1}
\frac{d}{dx} \log[x]_q=-\log(q) \frac{q^x}{1-q^x},
\ee
we get $P_0(x)=1$. Differentiating \eqref{basicder}  shows that $P_{n+1}$ must indeed satisfy \eqref{eulerpoly}. Formula \eqref{psider} now follows easily from \eqref{psi defn} and \eqref{psi1}.
\end{proof}

\begin{rem}
The first few Eulerian polynomials are given by
\[
\begin{split}
&P_1(x)=1,\\
&P_2(x)=x+1,\\
&P_3(x)=x^2+4x+1,\\
&P_4(x)=x^3+11x^2+11x+1,\\
&P_5(x)=x^4+26x^3+66x^2+26x+1.
\end{split}
\]
They were studied (without being explicitly named) in \cite{Moak84} in connection with the $q$-analogue of Stirling's formula. It is easy to see from \eqref{eulerpoly} that $P_n(0)=1$ for all $n$. In \cite{Moak84} it was proved that all the coefficients are positive and that $P_n(1)=n!$. It follows that for $0<q<1$ we have
\be\label{maxeuler}
1<P_n(q^x)<1+c_{n}q^x,
\ee
where $c_n:=n!-1$.  
\end{rem}
To simplify some of our statements we will set

\be\label{fnder}
F_n(x):=F_n(x;q):=\frac{d^n}{dx^n}\log f(x;q).
\ee

\begin{prop}\label{quotientder}
Let $f(x)$ and $F_n(x)$ be as in \eqref{frs} and \eqref{fnder}, respectively and assume $0<q<1$. For each $n\geq 2$ there exists $M_n\geq 0$ such that 
$F_n(x)$ is monotone for $x>M_n$. More specifically, if $\Upsilon$ is as in \eqref{upsilon} then $F_n(x)$ has the same sign as $\Upsilon\log^n q$.
\end{prop}

\begin{proof}
Using \eqref{psider} we see that
\be
\begin{split}
F_n(x)&=\frac{d^{n-1}}{dx^{n-1}}\left(\sum_{i=1}^r \psi_q(x+a_i)-\sum_{j=1}^s \psi_q(x+b_j)\right)\\
&=\log^n\/q\left[\sum_{i=1}^{r}\sum_{m=0}^{\infty}\dfrac{q^{x+a_i+m}P_{n-1}(q^{x+a_i+m})}{(1-q^{x+a_i+m})^n}-\sum_{j=1}^{s}\sum_{m=0}^{\infty}\dfrac{q^{x+b_j+m}P_{n-1}(q^{x+b_j+m})}{(1-q^{x+b_j+m})^n}\right].
\end{split}
\ee

Since $0<q<1$, we see that for all $m\geq 0$ and $a>0$ we have, using \eqref{maxeuler}  
\[
q^{x+a+m}<\frac{q^{x+a+m}P_{n-1}(q^{x+a+m})}{(1-q^{x+a+m})^n}<  \frac{(1+c_{n-1}q^{x+a})q^{x+a+m}}{(1-q^{x})^n}.
\] 
Summing over $m\geq 0$ and using \eqref{psider} we get
\be\label{psix+a}
\frac{q^{x+a}}{1-q}<\log^{-n}q\dfrac{d^{n-1}}{dx^{n-1}}\psi_q(x+a)<\frac{(1+c_{n-1}q^{x+a})\, q^{x+a}}{(1-q)(1-q^{x})^n}.
\ee
Thus
\be\label{sign1}
q^x(\Upsilon-h_B(q^x))<\frac{1-q}{\log^{n} q}F_n(x)<q^x(h_A(q^x)+\Upsilon),
\ee
where
\[
\begin{split}
h_A(x)&=\sum_{i=1}^rq^{a_i}\left(\frac{1+c_{n-1}q^{x+a_i}}{(1-q^x)^n}-1 \right),\\
h_B(x)&=\sum_{j=1}^s q^{b_j}\left(\frac{1+c_{n-1}q^{x+b_j}}{(1-q^x)^n}-1\right).
\end{split}
\]
Since 
\[
\lim_{x\to +\infty}h_A(q^x)=\lim_{x\to +\infty}h_B(q^x)=0,
\]
then there exists $M_n\geq 0$ such that $\max(\abs{h_A(q^x)},\abs{h_B(q^x)})<\frac{|\Upsilon|}{2}$ for all $x>M_n$. It follows that for such $x$, all three terms in \eqref{sign1} must have the same sign as $\Upsilon$, and the result follows.
\end{proof}

\begin{rem}
It follows from Proposition \ref{quotientder} that for $n\geq 1$ and $0<q<1$, either $F_n(x;q)$ or $-F_n(x;q)$ is completely monotone for sufficiently large $x$.
\end{rem}

Using \eqref{gammaq-1} we can extend Propositions \ref{quotientdecrease} and \ref{quotientder} to the case $q>1$ as follows.

\begin{prop}
Assume $q>1$, and let $f(x;q)$ and $F_n(x;q)$ be as in \eqref{frs} and \eqref{fnder}, respectively. Then $f(x;q)$ and $F_n(x;q)$ are monotone for sufficiently large $x$. More specifically we have
\begin{enumerate}
\item both $f$ and $F_1$ are increasing (resp. decreasing) if $r>s$ (resp. $r<s$),

\item if we set 
\[
\Upsilon^*:=\sum_{j=1}^s q^{-bj}-\sum_{i=1}^r q^{-a_i},
\]
then, for $n\geq 3$, $F_n(x;q)$ has the same sign as $\Upsilon^*(\log q)^n$. Furthermore, if $r=s$ then the same is true for $F_1$ and $F_2$. 
\end{enumerate} 
\end{prop}

\begin{proof}
From \eqref{gammaq-1} we see that
\[
\log f(x;q)=\left(\sum_{i=1}^r (x+a_i-1)(x+a_i-2)-\sum_{j=1}^s (x+b_j-1)(x+b_j-2)\right)\log\sqrt{q} +\log f(x;q^{-1}).
\]
It easily follows that
\be\label{array1}
\begin{split}
F_1(x;q)&=\left(\sum_{i=1}^r (2x+2a_i-3)-\sum_{j=1}^s (2x+2b_j-3)\right)\log\sqrt{q} +F_1(x;q^{-1}),\\
&=(r-s)x\log q+ \left(\sum_{i=1}^r (2a_i-3)-\sum_{j=1}^s (2b_j-3)\right)\log\sqrt{q} +F_1(x;q^{-1}),
\end{split}
\ee
\be\label{arr2}
F_2(x;q)=(r-s)\log{q} +F_2(x;q^{-1}),
\ee
\be\label{arr3}
F_n(x;q)=F_n(x;q^{-1}), \, \textrm{ for all } n\geq 3.
\ee

Assume $0<p<1$. From \eqref{logf'} we see that, 
\[
\lim_{x\to \infty} F_1(x;p)=(s-r)\log p,
\]
and it follows from \eqref{array1} that for sufficiently large $x$, $F_1(x;q)$ has the same sign as $(r-s)$.

from \eqref{sign1} we see that 
\[
\lim_{x\to \infty}F_2(x;p)=0,
\]
and it follows from \eqref{arr2} that $F_2(x;q)$ also has the same sign as $(r-s)$. This proves part (1) of the proposition. Part (2) follows easily from \eqref{arr3} and the corresponding statements in Propositions \ref{quotientder} and \ref{quotientdecrease}.
\end{proof}

\begin{rem}
In \cite[Corollary 3.5]{GrinshIsm} it was proved that for $\alpha, \beta, \lambda >0$, and $0<q<1$ the function
\be
G(x):=\log \frac{\Gamma_q(x)\Gamma_q(x+\alpha+\beta)\Gamma_q(x+\alpha+\lambda)\Gamma_q(x+\beta+\lambda)}{\Gamma_q(x+\alpha)\Gamma_q(x+\beta)\Gamma_q(x+\lambda)\Gamma_q(x+\alpha+\beta+\lambda)}
\ee
is completely monotone. As an illustration of our results above, we note that this would follow from Propositions \ref{quotientdecrease} and \ref{quotientder} as it corresponds to $r=s=4$ and 
\[
\Upsilon=(1-q^\alpha)(1-q^\beta)(1-q^\lambda) >0,
\]
and hence indeed $\frac{d^n}{dx^n}G(x)$ will have the same sign as $\log^n q$. If one considers functions as in \eqref{frs} corresponding to $\Upsilon$ of the form 
\[
\Upsilon=\prod_{k=1}^t (1-q^{\alpha_k}),
\]  
then we can obtain many generalizations of that Corollary. It is even possible to consider more general factors than $(1-q^{\alpha_k})$ and obtain even more general results.
\end{rem}

Next, we also derive an integral representation formula for $\psi_q$ using the following formula of Plana, see~\cite{aar}.

\begin{thm}[Plana]\label{thm:plana extension}
If $m$  and $n$ are positive integers, $m\leq n$, and $\phi(z)$ is a function which is  analytic  and bounded for all values of $z$ such that
$m\leq \Re\phi(z)\leq n$, then
\begin{equation}
\begin{split}
\sum_{k=m}^{n}\phi(k)=\dfrac{\phi(m)+\phi(n)}{2}+\int_{m}^{n}\phi(x)\,dx-i\int_{0}^{\infty}\dfrac{\phi(n+iy)-\phi(m+iy)-\phi(n-iy)+\phi(m-iy)}{e^{2\pi\,y-1}}\,dy.
\end{split}
\end{equation}
Moreover, if $\phi(z)$ is an entire function such that $\phi(n\pm iy)\to 0$  uniformly in $y$, then
\begin{equation}
\sum_{k=0}^{\infty}\phi(k)=\frac{1}{2}\phi(0)+\int_{0}^{\infty}\phi(x)\,dx+i\int_{0}^{\infty}\dfrac{\phi(iy)-\phi(-iy)}{e^{2\pi\,y-1}}\,dy.
\end{equation}
\end{thm}

\begin{thm}
 For  fixed $q\in(0,1)$  we have for $x>0$
\begin{equation}\label{psi representation formula}
\begin{split}
\psi_q(x)=&-\log(1-q)+\frac{\log\,q}{2}\frac{q^{x}}{1-q^{x}}+\log(1-q^{x})\\
&-2q^{x}\log\,q\int_{0}^{\infty}\dfrac{\sin(t\log\,q)}{(e^{2\pi\,t}-1)\left(1-2q^{x}\cos(t\log\,q)+q^{2x}\right)}\,dt.
\end{split}
\end{equation}

\end{thm}

\begin{proof}
Set $g(x)=\frac{1}{\log q}\left(\psi_q(x)+\log(1-q)\right)$. Then from \eqref{psi defn},
\[g(x)=\sum_{k=0}^{\infty}\phi_x(k),\quad \phi_x(k)=\frac{q^{x+k}}{1-q^{x+k}}.\]
Clearly $\lim_{n\to\infty}\phi_x(n)=0$   and $\lim_{n\to\infty}\phi_x(n+iy) =0$  uniformly for $y>0$. Hence applying Theorem \ref{thm:plana extension} gives
\begin{equation*}
\begin{split}
g(x)&=\frac{1}{2}\dfrac{q^{x}}{1-q^{x}}+\int_{0}^{\infty}\dfrac{q^{x+t}}{1-q^{x+t}}\,dt
+i\int_{0}^{\infty}\frac{1}{e^{2\pi\,t-1}}\left(\dfrac{q^{x+it}}{1-q^{x+it}}-\dfrac{q^{x-it}}{1-q^{x-it}}\right)\,dt\\
&=\frac{1}{2}\dfrac{q^{x}}{1-q^{x}}+\frac{1}{\log\,q}\log(1-q^x)-+i\int_{0}^{\infty}\dfrac{q^{x}(q^{it}-q^{-it})}{(e^{2\pi\,y}-1)\left(1-q^{x}(q^{it}+q^{-it})+q^{2x}\right)}\,dt.
\end{split}
\end{equation*}
Using that
\begin{eqnarray*}
q^{it}+q^{-it}&=&e^{it\log\,q}+e^{-it\log\,q}=2\cos(t\log q),\\
q^{it}-q^{-it}&=&e^{it\log\,q}-e^{-it\log\,q}=2i\sin(t\log q),
\end{eqnarray*}
we obtain \eqref{psi representation formula} and completes the proof.
\end{proof}

\begin{cor}
For  fixed $q>1$  we have for $x>0$
\be\label{PRF} \begin{gathered}
\psi_q(x)=-\frac{1}{2}\log q-\log(q-1)+\frac{\log\,q}{2}\frac{q^{x}}{q^{x}-1}+\log(q^{x}-1)\\-2q^{x}\log\,q\int_{0}^{\infty}\dfrac{\sin(t\log\,q)}{(e^{2\pi\,t}-1)\left(1-2q^{x}\cos(t\log\,q)+q^{2x}\right)}\,dt.
\end{gathered}
\ee
\end{cor}

\begin{rem}
From \eqref{psi representation formula} and \eqref{PRF} we conclude that
\[\psi_q(x)=\log [x]_q+\frac{q^x}{1-q^x} \frac{\log\/q}{2}+o(1),\]
as $x\to\infty$, when $0<q<1$, and 
\[\psi_q(x)=\log[x_q]+\frac{1}{q^{x}-1}+\frac{\log\/q}{2}+o(1),\]
as $x\to\infty$, when $q>1$.
\end{rem}


\section{\bf The functional equation $f(x+1)=[x]_qf(x)$ and its generalizations}
Throughout this section let $k\in \R$, $q,w>0$, and $a_1,\dots,a_r, b_1,\dots, b_s \geq 0$. We consider the functional equation and initial condition
\be\label{maineqn1}
\begin{split}
f(x+w)&=\left(\frac{[x+a_1]_q \cdots [x+a_r]}{[x+b_1]_q\cdots [x+b_s]_q}\right)^k f(x),\\
f(w)&=1.
\end{split}
\ee
It is straightforward to check that the function $F$ given by
\be\label{mainsol1}
F(x)=\left(\prod_{j=1}^{r}\dfrac{\Gamma_{q^{w}}(\frac{x+a_j}{w})}{\Gamma_{q^{w}}(\frac{w+a_j}{w})}\prod_{i=1}^{s}
\dfrac{\Gamma_{q^{w}}(\frac{w+b_i}{w})}{\Gamma_{q^{w}}(\frac{x+b_i}{w})}[w]_q^{(r-s)(\frac{x}{w}-1)}\right)^k
\ee
is a solution of \eqref{maineqn1}. We are interested in stating and proving a number of additional conditions on \eqref{maineqn1} that will ensure that $F$ is its unique solution.

We start with a lemma which shows that the problem could always be reduced to the case $w=1$.

\begin{lem}\label{wto1k}
Let $f$ be any function defined for $x>0$, and consider
\be\label{qktrans}
g(x):=[w]_q^{k(r-s)(1-x)}f(wx).
\ee
Then $f$ satisfies the functional equation \eqref{maineqn1} if and only if $g$ satisfies 
\be \label{g1w}
g(x+1)=\left(\frac{[x+\frac{a_1}{w}]_{q^w}\cdots [x+\frac{a_r}{w}]_{q^w}}{[x+\frac{b_1}{w}]_{q^w}\cdots[x+\frac{b_s}{w}]_{q^w}}\right)^k g(x).
\ee

\end{lem}
\begin{proof}
We clearly have $g(1)=f(w)$. If $f$ satisfies \eqref{maineqn1} then we have
\[
\begin{split}
g(x+1)&=[w]_q^{-k(r-s)x}f(wx+w)\\
&=[w]_q^{-k(r-s)x}\left(\frac{[wx+a_1]_q\cdots [wx+a_r]_q}{[wx+b_1]_q\cdots[wx+b_s]_q}\right)^kf(wx)\\
&=\left(\prod_{i=1}^r\frac{1-q^{wx+a_i}}{1-q^w}\prod_{j=1}^s \frac{1-q^w}{1-q^{wx+b_j}}\right)^k [w]_q^{k(r-s)(1-x)}f(wx)\\
&=\left(\frac{[x+\frac{a_1}{w}]_{q^w}\cdots [x+\frac{a_r}{w}]_{q^w}}{[x+\frac{b_1}{w}]_{q^w}\cdots[x+\frac{b_s}{w}]_{q^w}}\right)^k g(x),
\end{split}
\]
proving necessity. Sufficiency follows in  the same fashion using the fact that
\be
f(x)=[w]_q^{k(r-s)(\frac{x}{w}-1)}g\left(\frac{x}{w}\right),
\ee 
\end{proof}

We also establish a connection between the cases $q>1$ and $0<q<1$.

\begin{lem}\label{qrecipr}
Let $f$ be a function defined for $x>0$. For $w,\, q>0$ and $a\geq 0$ consider the transformation
\be\label{qtrans}
h(x):= T_{q,a,w}(f)(x):=q^{\frac{-(x-w)(x+2a-2)}{2w}}f(x).
\ee
Then $f$ satisfies the functional equation
\[
f(x+w)=[x+a]_qf(x)
\]
if and only if $h$ satisfies 
\[
h(x+w)=[x+a]_{q^{-1}} h(x).
\]
Furthermore, we have $h(w)=f(w)$.
\end{lem}
\begin{proof}
We have
\be\label{lemqrec}
\begin{split}
f(x+w)&=\frac{1-q^{x+a}}{1-q}f(x)=q^{x+a-1}\frac{q^{-x-a}-1}{q^{-1}-1}f(x)\\
&=q^{x+a-1}[x+a]_{q^{-1}}f(x).
\end{split}
\ee
Multiplying \eqref{lemqrec} by $q^{\frac{-x(x+w+2a-2)}{2w}}$ gives
\[
h(x+w)=[x+a]_{q^{-1}}h(x),
\]
where the last equality follows since
\[
x(x+w+2a-2)-2w(x+a-1)=x^2+(2a-2-w)x-w(2a-2)=(x-w)(x+2a-2).
\]
Thus we have proved necessity. Sufficiency follows at once since $T_{q^{-1},a,w}(T_{q,a,w}(f))=f$. 
\end{proof}

Lemma \ref{qrecipr} could be extended as follows.

\begin{cor}\label{qreciprs}
Let $f$ be a function defined for $x>0$. For  $w,\, q>0$, $k \in \R$, and nonnegative real numbers $a_1,\dots, a_r, b_1,\dots, b_s$, set $A:=\sum_{i=1}^r a_i$, $B:=\sum_{i=1}^s b_i$, and
\be\label{gf}
h(x):=q^{\frac{k(w-x)}{2w}((r-s)(x-2)+2(A-B))} f(x).
\ee
Then $f$ satisfies the functional equation
\[
f(x+w)=\left(\frac{[x+a_1]_q\cdots [x+a_r]_q}{[x+b_1]_q\cdots [x+b_s]_q}\right)^kf(x)
\]
if and only if $h$ satisfies 
\[
h(x+w)=\left(\frac{[x+a_1]_{q^{-1}}\cdots [x+a_r]_{q^{-1}}}{[x+b_1]_{q^{-1}}\cdots [x+b_s]_{q^{-1}}}\right)^k h(x).
\]
Furthermore, we have $h(w)=f(w)$.
\end{cor}
\begin{proof}
The proof is similar to Lemma \ref{qrecipr} and is omitted for brevity.
\end{proof}

\begin{thm}\label{KMgeneral}
Let $f(x)$ be a positive function defined on $(0,\infty)$  that satisfies
\begin{itemize}
\item[(i)]  For $x>0$, $q\in(0,1)$, $k \in \R$, and $a_1,\dots, a_r, b_1,\dots, b_s \geq 0$
\[
f(x+1)=\left(\frac{[x+a_1]_q \cdots [x+a_r]_q}{[x+b_1]_q\cdots [x+b_s]_q}\right)^k f(x),
\]
\item[(ii)] $(1-q)^{k(r-s)x} f(x)$ is a monotone  function for $x>M$, where $M>0$ is a constant,
\item[(iii)] $f(1)=1$.
\end{itemize}
Then for all $x>0$
\be
f(x)=\left(\prod_{i=1}^r\frac{\Gamma_q(x+a_i)}{\Gamma_q(1+a_i)}\prod_{j=1}^s\frac{\Gamma_q(1+b_j)}{\Gamma_q(x+b_j)}\right)^k.
\ee

\end{thm}

\begin{proof}
By (i) we see that $f$ is completely defined by its values on $(0,1]$, so it suffices to prove the theorem on that interval. Let $n$ be a positive integer greater than $M$ and let $0<x\leq 1$. Then
\[n<n+x\leq n+1,\]
and from  the condition (ii) of the theorem, we obtain either
\be\label{decrease1}
(1-q)^{k(r-s)n}f(n)\geq (1-q)^{k(r-s)(n+x)}f(n+x)\geq (1-q)^{k(r-s)(n+1)}f(n+1),
\ee
or
\be\label{increase1}
(1-q)^{k(r-s)n}f(n)\leq (1-q)^{k(r-s)(n+x)}f(n+x)\leq (1-q)^{k(r-s)(n+1)}f(n+1).
\ee
Using the conditions (i)  and (iii), we see that \eqref{decrease1} becomes
\be\label{ineq1}
\left(\frac{\prod(q^{1+a_i};q)_{n-1}}{\prod(q^{1+b_j};q)_{n-1}}\right)^k\geq \left( (1-q)^{(r-s)(x-1)}\frac{\prod(q^{x+a_i};q)_{n}}{\prod(q^{x+b_j};q)_{n}}\right)^k f(x) \geq \left(\frac{\prod(q^{1+a_i};q)_{n}}{\prod(q^{1+b_j};q)_{n}}\right)^k.\; 
\ee
Calculating the limit as $n\to\infty$ in \eqref{ineq1} yields the theorem in the decreasing case. The proof for the increasing case (corresponding to \eqref{increase1}) is identical.
\end{proof}

\begin{cor}\label{KMgeneralw}
Let $f(x)$ be a positive function defined on $(0,\infty)$  that satisfies \eqref{maineqn1} for some $q\in (0,1)$. If  
\[
(1-q)^{(r-s)x}[w]_q^{k(r-s)(1-x)} f(wx)
\]
is  monotone  for $x>M\geq0$, then $f$ must be given by \eqref{mainsol1}.
\end{cor}

\begin{proof}
Let $g(x):=[w]_q^{k(r-s)(1-x)} f(wx)$, then by Lemma \ref{wto1k} we see that $g$ satisfies \eqref{g1w}, and the result follows after a straightforward computation from  Theorem \ref{KMgeneral}.
\end{proof}

\begin{cor}\label{KMgeneralwq1}
Let $f(x)$ be a positive function defined on $(0,\infty)$  that satisfies \eqref{maineqn1} for some $q\in (0,1)$. Let $h$, $A$, and $B$ as in Corollary \textup{\ref{qreciprs}}, and assume that 
\[
(1-q^{-1})^{(r-s)x}[w]_{q^{-1}}^{k(r-s)(1-x)} h(wx)
\]
is  monotone  for $x>M\geq0$, then $f$ must be given by \eqref{mainsol1}.
\end{cor}

\begin{proof}
The proof follows from a straightforward application of Corollaries \ref{qreciprs} and \ref{KMgeneralw}, and we leave the details to the reader. 
\end{proof}

\begin{thm}\label{logconvexwab}
Let  $f$ be a positive function defined on $(0,\infty)$ that satisfies \eqref{maineqn1} and is  logarithmically convex  or logarithmically concave for all $x>M\geq 0$.  Then $f$ must be given by \eqref{mainsol1}.

\end{thm}

\begin{proof}
By iterating \eqref{maineqn1} $n$ times we get
\be\label{x+nw}
f(x+nw)=\left[(1-q)^{n(s-r)}\frac{\prod_{i=1}^r(q^{x+a_i};q^w)_n}{\prod_{j=1}^s(q^{x+b_j};q^w)_n}\right]^k f(x),\;  x>0,\;  n\geq 0.
\ee
Moreover, since $f(w)=1$ we get
\be\label{nw}
f((n+1)w)=\left[(1-q)^{n(s-r)}\frac{\prod_{i=1}^r(q^{w+a_i};q^w)_n}{\prod_{j=1}^s(q^{w+b_j};q^w)_n}\right]^k,\;  x>0,\;  n\geq 0.
\ee
For $0<x\leq w$ and $n\geq 1$ we have
\[
(n-1)w<nw<nw+x\leq (n+1)w. 
\]
Let $n$ be large enough so that $(n-1)w>M$. If $f$ is logarithmically convex we get
\be\label{convexw}
\dfrac{\log\/f(nw)-\log\/f((n-1)w)  }{w}\leq \frac{\log\/f(nw+x)-\log\/f(nw) }{x}\leq \dfrac{\log f((n+1)w)-\log f(nw) }{w},
\ee
whereas if it is logarithmically concave 
\be\label{concavew}
\dfrac{\log\/f(nw)-\log\/f((n-1)w)  }{w}\geq \frac{\log\/f(nw+x)-\log\/f(nw) }{x}\geq \dfrac{\log f((n+1)w)-\log f(nw) }{w}.
\ee
By \eqref{x+nw} and \eqref{nw} we see that
\be\label{middleterm}
\log f(x+nw)-\log f(nw)=k\log\left( (1-q)^{s-r} \prod_{i=1}^r \frac{(q^{x+a_i};q^w)_n}{(q^{w+a_i};q^w)_{n-1}}\prod_{j=1}^s \frac{(q^{w+b_j};q^w)_{n-1}}{(q^{x+b_j};q^w)_n}f(x) \right).
\ee
Since 
\be
\lim_{n\to \infty} \log f((n+1)w)-\log f(nw)=\lim_{n\to \infty} \log f(nw)-\log f((n-1)w)=k(s-r)\log(1-q).
\ee
It follows from \eqref{convexw} and \eqref{concavew} that 
\be\label{squeeze}
\lim_{n\to \infty} \log f(x+nw)-\log f(nw)=\frac{kx(s-r)}{w}\log(1-q).
\ee
By \eqref{middleterm} this translates into
\be\label{forcedfq}
f(x)=\left[(1-q)^{(s-r)(\frac{x}{w})-1} \prod_{i=1}^r \frac{(q^{w+a_i};q^w)_\infty}{(q^{x+a_i};q^w)_{\infty}}\prod_{j=1}^s \frac{(q^{x+b_j};q^w)_{\infty}}{(q^{w+b_j};q^w)_\infty} \right]^k,
\ee
which indeed is equivalent to \eqref{mainsol1}, completing the proof.
\end{proof}

\begin{cor}\label{moakgeneral}
 Let  $f$ be a positive function defined on $(0,\infty)$ that satisfies \eqref{maineqn1} for some $q>1$. If
 $q^{\frac{(r-s)(w-x)(x-2)}{2w}}f$ is a   logarithmically convex  or logarithmically concave for all $x>M\geq 0$, then $f$ must be given by \eqref{mainsol1}.

\end{cor}
\begin{proof}
Let $A$, $B$, and $h$ be as in Corollary \ref{qreciprs}, then by that Corollary we have 
that $h$ satisfies

\be\label{fe1q1g}
h(x+w)=\left\{\dfrac{[x+a_1]_{q^{-1}}\ldots [x+a_r]_{q^{-1}}}{[x+b_1]_{q^{-1}}\ldots [x+b_{s}]_{q^{-1}}}\right\}^k h(x).
\ee
Furthermore, by the condition  above, we see that $g(x):=q^{\frac{(x-w)(A-B)}{w}}h(x)$ is either logarithmically concave or logarithmically convex for large enough $x$. The same must be true for $h(x)$ as $\log g$ and $\log h$ differ only by a linear function of $x$, which doesn't change the concavity behavior. From Theorem \ref{logconvexwab} we get that 
\be\label{gq1}
h(x)=\left(\prod_{j=1}^{r}\dfrac{\Gamma_{q^{-w}}(\frac{x+a_j}{w})}{\Gamma_{q^{-w}}(\frac{w+a_j}{w})}\prod_{i=1}^{s}
\dfrac{\Gamma_{q^{-w}}(\frac{w+b_i}{w})}{\Gamma_{q^{-w}}(\frac{x+b_i}{w})}[w]_{q^{-1}}^{(r-s)(\frac{x}{w}-1)}\right)^k,
\ee
and the result follows from \eqref{gq1} and \eqref{gf}.
\end{proof}

\begin{rem}
Corollary \ref{moakgeneral} generalizes Theorem 2 in \cite{Moak}, not only in having a more general functional equation, but also because we don't assume the differentiability of $f$. Specifically if we set  $w=1, k=1, s=0, r=1, a_1=0$ and assume that $f$ is twice differentiable, then condition (i) of Corollary \ref{moakgeneral} translates into 
\[
\frac{d^2}{dx^2}\left(\frac{(1-x)(x-2)}{2}\log q+\log f\right)\geq 0,
\]
which is equivalent to the condition (2.15) on p. 282 of \cite{Moak}.
\end{rem}

The following result, proved by John in \cite{John}, will enable us to give a vast generalization of Theorem 1 in \cite{Moak}.

\begin{lem}\label{lem1} {John \cite[Theorem A]{John}}
 Let $g(x)$ be defined for $x > 0$, and let
\be\label{infg0}
\inf_{x>0} g(x) = 0,
\ee
then any two monotone non-decreasing solutions of
\be\label{diff1}
f(x + 1)-f(x) = g(x),\;\;x>0
\ee
differ at most by a constant.
\end{lem}

We have the following generalizations of Lemma \ref{lem1}.

\begin{lem}\label{lem2}
 Let $g(x)$ be defined for $x > 0$, and for fixed $w>0$ consider solutions of the difference equation
\be\label{diffw}
f(x + w)-f(x) = g(x),\;\;x>0.
\ee
\begin{enumerate}
\item If for some $L\geq0$
\be\label{infgleq0}
\inf_{x>L} g(x) \leq 0,
\ee
then any two solutions of \eqref{diffw} that are non-decreasing for $x>M$ differ at most by a constant.
\item If for some $L\geq 0$
\be
\sup_{x>L} g(x) \geq 0,
\ee
then any two solutions of \eqref{diffw} that are non-increasing for $x>M$ differ at most by a constant.
\end{enumerate}

\end{lem}

\begin{proof}
The proof is essentially similar to the proof of Lemma \ref{lem1} in \cite[p. 176]{John}, however we present it to highlight that the weaker condition \eqref{infgleq0} can replace \eqref{infg0} without weakening the result. To prove (1), let $f_1(x), f_2(x)$ be two solutions of \eqref{diffw} for $x>0$ that are non-decreasing for $x>M$ and set 
\[
F(x):=f_1(x)-f_2(x).
\]
From \eqref{diffw}, we see that for all $x>0$,  $F(x+w)=F(x)$. Let $N$ be a positive integer such that $Nw>L$.  Since $f_1$ and $f_2$ are non-decreasing, we see that for $Nw\leq x\leq (N+1)w$ we have
\[
f_1(Nw)-f_2((N+1)w)\leq F(x)\leq f_1((N+1)w)-f_2(Nw).
\]
Thus $F(x)$ is bounded on the interval $[Nw,(N+1)w]$, and hence, by the periodicity, bounded for all $x>0$. Set
\[
M:=\sup_{x>0}F(x),\, \, \, m:=\inf_{x>0}F(x). 
\]
If $F(x)$ is not constant, set $0<\epsilon:=\frac{M-m}{3}$. By \eqref{infgleq0}, we can find $x_0>L$ such that $g(x_0)\leq \epsilon$. Consider $a, b$ such that
\[
L<x_0\leq a \leq b \leq x_0+w.
\]
Using that $f_1$ and $f_2$ are non-decreasing for $x>L$ we get
\[
\begin{split}
\epsilon\geq g(x_0)&=f_1(x_0+w)-f_1(x_0)\geq f_1(b)-f_1(a)\\
&=f_2(b)-f_2(a)+F(b)-F(a)\geq F(b)-F(a).
\end{split}
\]
First set $a=x_0$. We get for all $b\in [x_0,x_0+w]$
\[
\epsilon \geq F(b)-F(x_0).
\]
Hence 
\be\label{epsM}
\epsilon \geq \sup_{b\in [x_0,x_0+w]} (F(b)-F(x_0))=M-F(x_0).
\ee
Setting $b=x_0+w$ we likewise get
\be\label{epsm}
\epsilon \geq \sup_{a\in [x_0,x_0+w]} (F(x_0+w)-F(a))=F(x_0)-m.
\ee
Adding \eqref{epsM} and \eqref{epsm} we get the contradiction
\[
2\epsilon=2\frac{(M-m)}{3}\geq M-m.
\]
This proves that $F(x)$ must be a constant, completing the proof of (1). Statement (2) follows from (1) by noticing that if $h$ a non-increasing solution of \eqref{diffw}, then $-h$ is a non-decreasing solution of 
\[
f(x+w)-f(x)=-g(x),
\] 
and that $\inf(-g)=-\sup(g)$.
\end{proof}

\begin{cor}\label{corg}
Let $g(x)$ be a function defined for $x>0$ such that $\lim_{x\to \infty}g(x)=0$. Then any two solutions of \eqref{diffw} that are monotone for large enough $x$ differ at most by a constant.
\end{cor}

\begin{thm}
Let $f$ be a positive solution of \eqref{maineqn1}, and assume that for some $n\geq 1$, $\frac{d^n}{dx^n}\log f$ is monotone for $x>L\geq 0$, then $f$ must be given by $F$ in \eqref{mainsol1}.
\end{thm}

\begin{proof}
Consider the functional equation
\be\label{nder}
h(x+w)-h(x)=g(x),
\ee
where 
\[
g(x):=k\frac{d^n}{dx^n}\left(\sum_{i=1}^r \log(1-q^{x+a_i})-\sum_{j=1}^s \log(1-q^{x+b_j})\right).
\]
We easily see that $\frac{d^n}{dx^n}\log f$ is an eventually monotone solution of 
\eqref{nder}. Using \eqref{basicder} we see that $\lim_{x\to \infty}g(x)=0$,  and thus by Corollary \ref{corg} we get that $\log f$ and $\log F$ differ at most by a polynomial in $x$. However, since the initial condition and the functional equation imply that $f$ and $F$ agree on all multiples of $w$, we see that that polynomial has infinitely many roots and hence must be identically zero and hence $f=F$.
\end{proof}

\section{{\bf The functional equation $1/f(x+1)=[x]_q f(x)$ and its generalizations}}
Mayer~\cite{Mayer} proved that  the function 
\begin{equation}
G(x)=\dfrac{1}{\sqrt{2}}\dfrac{\Gamma(x/2)}{\Gamma((x+1)/2)},\;x>0,
\end{equation}
   is  logarithmically convex  and decreasing on $(0,\infty)$  and satisfies the functional equation
\be\label{Mayer1} \frac{1}{f(x+1)}=xf(x),\; f(1)=1,\;x>0.\ee
He also proved  that we need only $f$ to be  convex, decreasing, and positive on $(0,\infty)$ to prove the   uniqueness of the solution of the functional equation \eqref{Mayer1}. 

In this section, we study the existence and uniqueness of solutions of 
\be\label{Geqn1}
\frac{1}{f(x+1)}=[x]_q f(x),\;f(1)=1,\;x>0,\; q>0,
\ee
and some generalizations of it. Let $k>0$ be a real number and consider the functional equation
\be\label{Geqn}
\frac{1}{f(x+1)}=[x]^k_q f(x),\;f(1)=1,\;x>0.
\ee
Set
\be\label{ZM1}\begin{split}G_q(x)&:=\dfrac{1}{\sqrt{[2]_q}}\dfrac{\Gamma_{q^2}(x/2)}{\Gamma_{q^2}((x+1)/2)}\\
&=\left\{\begin{array}{cc}\sqrt{1-q}\dfrac{\left(q^{x+1};q^2\right)_{\infty}}{\left(q^x;q^2\right)_{\infty}},&\;0<q<1,\\
G(x),& \; q=1,\\
\sqrt{1-q^{-1}}\dfrac{\left(q^{-x-1};q^{-2}\right)_{\infty}}{\left(q^{-x};q^{-2}\right)_{\infty}},&\;q>1.
\end{array}\right.\end{split}\ee
One can verify that $G_q^k(x)$ is a solution of the functional equation \eqref{Geqn}.
Since for all $x>0$ and  $q>0$ we have
\[
G_q(x)=G_{q^{-1}}(x), 
\]
it suffices to study the case $q \in (0,1]$. The classical case $q=1$ was covered by several authors (see \cite{Mayer} and \cite{Thielman} for instance), so we focus here on $q\in (0,1)$.

\begin{lem}
The function $G_q(x)$, $0<q<1$, is strictly decreasing and logarithmically convex for $x\in (0,\infty)$. 
\end{lem}

\begin{proof}
Using the notation of Propositions \ref{quotientdecrease} and \ref{quotientder}, we see that 
\[
\Upsilon=1-q>0,
\]
and the result follows from the third case of Proposition \ref{quotientdecrease} ($r=s=1$) and from the case $n=2$ of Proposition \ref{quotientder}.
\end{proof}
It follows that for all $k>0$, $G_q^k(x)$ is also strictly decreasing and logarithmically convex. Since the logarithmic convexity implies convexity, we conclude that $G^k_q(x)$ is a convex function for any positive $k$. 

Next, we give estimates for the function $G_q(x)$. Since the function $G_q(x)$ is decreasing, we have 
\[G_q(x)>G_q(x+1)>G_q(x+2).\]
Multiplying  the previous equation by $G_q(x+1)$ and using that $G_q(x)$ is a solution of \eqref{Geqn}, we obtain
\be\label{ZM3} \frac{1}{[x]_q}>G_q^2(x+1)>\frac{1}{[x+1]_q}.\ee

\begin{thm}\label{Mayeranalogue} If $f(x)$ is convex for all $x> M\geq 0$ and  satisfies \eqref{Geqn}, then 
$f(x)=G^k_q(x)$, for all $x>0$. 
\end{thm}

\begin{proof}
Assume that $f(x)$ is another solution of \eqref{Geqn} and set 
\[Q(x)=\dfrac{f(x)}{G^k_q(x)},\;x>0.\]
Then 
\be\label{ZM6}
Q(x+2n+1)=\dfrac{1}{Q(x)},\quad Q(x+2n)=Q(x),\;x>0,
\ee
where $n$ is a positive integer.  We  must get  $Q(x)=1$ for all $x>0$.  Note that \eqref{Geqn} implies that $f(x)$, and hence $Q(x)$ is never zero.  If $Q(x)>0$ for some particular value of $x$, then according  to \eqref{ZM3} and \eqref{ZM6} we obtain   
\be
\begin{split}
f(x+2n-1)&=\dfrac{1}{Q(x)}G^k_q(x+2n-1)<\dfrac{1}{Q(x)}\frac{1}{\sqrt{[x+2n-2]^k_q}},\\
f(x+2n)&=Q(x)G^k_q(x+2n)>Q(x)\dfrac{1}{\sqrt{[x+2n]^k_q}}. \\
\end{split}
\ee
We can choose $n$ large enough such that $x+2n>M$. Hence  
applying  the convexity of the function $f$, we obtain
\[f(x+2n)\leq \frac{1}{2}\left\{f(x+2n-1)+f(x+2n+1)\right\}.\]
Therefore,
\be\label{ZM7} Q^2(x)\leq \frac{1}{2}\left(\left(\frac{[x+2n+1]_q}{[x+2n-1]_q}\right)^\frac{k}{2}+1\right).\ee

Replacing $x$ by $x+1$ in the previous inequality and using \eqref{ZM6} gives 
\be\label{ZM8}\dfrac{1}{Q^2(x)}\leq\frac{1}{2}\left(\left(\frac{[x+2n+1]_q}{[x+2n-1]_q}\right)^\frac k 2+1\right).\ee

 Then calculating the limit as $n\to\infty$ gives $Q^2(x)=1$. 
If we assume that there exists an $x$ for which $Q(x) <0$, an argument similar
to the preceding one shows again that $Q^2(x)=1$.  That is 
\[G^k_q(x)=|f(x)|,\;\mbox{for all}\;x>0.\]
Next we determine the sign of $f(x)$. If we assume that there exists $x_0>0$ such that $f(x_0)=-G^k_q(x_0)$, then from the convexity of the function $f$,  we have
\[f(x_0+h)\leq \frac{1}{2}\left\{f(x_0)+f(x_0+2h)\right\}\leq\/\frac{1}{2}\left\{-G^k_q(x_0)+G^k_q(x_0+2h)\right\}, \]
for any $h>h_0$, where $h_0$  is chosen so that $x_0+h_0>K $. 
 Because $G^k_q(x)$ is decreasing, the right hand side of the previous inequality is negative. Hence $f(x)<0$ for all $x>x_0+h_0$. That is,
 \be \label{ZMf}f(x)=-G^k_q(x),\;\mbox{for all}\;x>x_0+h_0.\ee
This yields that $f$ is a smooth function and
\[
f''(x)=-\frac{d^2}{dx^2}G^k_q(x)<0,
\]
contradicting the convexity of $f$.
\end{proof}

The following uniqueness criterion of the solutions of the functional equation \eqref{Geqn} is a $q$-analogue of a result of Anastassiadis~\cite[p. 62]{anasta}. 

\begin{thm}\label{G1thm}
The only  function which   satisfies  
\be\label{eq:G0} f(x+1)\leq f(x),\quad\mbox{for all}\quad x>M\geq  0,\ee
\[f(x)\neq 0,\quad \mbox{for all}\quad x>0,\]
and the functional equation \eqref{Geqn},
is  $f(x)=G^k_q(x)$.
\end{thm}

\begin{proof}
From  \eqref{Geqn}, one can verify that
\bea
f(x+2n)&=&\prod_{j=0}^{n-1}\left(\dfrac{[x+2j]_q}{[x+2j+1]_q}\right)^k f(x),\\
f(x+2n-1)&=&\left(\dfrac{\prod_{j=0}^{n-2}[x+2j+1]_{q}}{\prod_{k=0}^{n-1}[x+2j]_q}\right)^k\dfrac{1}{f(x)},
\eea
 for all integer $n\geq 1$. Also from \eqref{eq:G0}
 \[f(x+2n-1)\geq f(x+2n) \geq f(x+2n+1),\quad\mbox{for all}\quad x>0,\]
 where we choose here $n$ large enough such that $x+2n-1>M$.
 This  gives that
\be
\left(\dfrac{\prod_{j=0}^{n-2}[x+2j+1]_{q}}{\prod_{j=0}^{n-1}[x+2j]_{q}}\right)^k \dfrac{1}{f(x)} \geq\left(\prod_{j=0}^{n-1}\dfrac{[x+2j]_{q}}{[x+2j+1]_{q}}\right)^k f(x)\geq\left(\dfrac{\prod_{j=0}^{n-1}[x+2j+1]_{q}}{\prod_{j=0}^{n}[x+2j]_{q}}\right)^k\dfrac{1}{f(x)}.
\ee
Hence
\be\label{eq:G3}
\left(\dfrac{1}{[x+2n-1]_{q}}\prod_{j=0}^{n-1}\dfrac{[x+2j+1]^2_{q}}{[x+2j]^2_{q}}\right)^k\dfrac{1}{f(x)}\geq f(x) \geq \left(\dfrac{1}{[x+2n]_{q}}\prod_{j=0}^{n-1}\dfrac{[x+2j+1]^2_{q}}{[x+2j]^2_{q}}\right)^k\dfrac{1}{f(x)}.
\ee
Then taking the limit as $n\to\infty$ in \eqref{eq:G3} gives
\[
f(x)=(1-q)^k\left(\dfrac{(q^{x+1};q^2)_{\infty}}{(q^{x};q^2)_{\infty}}\right)^{2k}\frac{1}{f(x)}.\]
Consequently,
\[f^2(x)=G_q^{2k}(x),\;\mbox{for all}\; x>0,\]
and hence $|f(x)|=G^k_q(x)$, for all $x>0$. We now prove that $f$ has no negative values. Suppose on the contrary that there exists $x_0>0$ such that 
$f(x_0)=-G^k_q(x_0)$. Then from the functional equation \eqref{Geqn}, we conclude that
\[f(x_0+2n)=-G^k_q(x_0+2n),\;\mbox{for all}\;n\in\mathbb{N}.\]
Let $n$ be large enough so that $x_0+2n>M$. Hence 
\[0\leq f(x_0+2n)-f(x_0+2n+1)\leq-G^k_q(x_0+2n)+G^k_q(x_0+2n+1)<0,\]
where the last inequality is strict since $G^k_q$ is strictly decreasing. This contradiction proves that $f(x)$ must be always equal to $G^k_q(x)$.  
\end{proof}

\begin{rem}
The assumption $f(x+1)\leq f(x)$ is weaker than assuming $f$ to be decreasing or eventually decreasing. For instance it is satisfied by $f(x)=\frac{\sin(2\pi x)}{x}$, which is obviously not eventually decreasing.
\end{rem}

Let $a_1,\dots, a_u, b_1, \dots, b_v \geq 0$, $w>0$, $k\in \R$. We now consider the more general functional equation

\be\label{eqfu}
f(x+w)=\left(\dfrac{[x+a_1]_q\ldots [x+a_u]_q}{[x+b_1]_q\ldots [x+b_v]_q}\right)^k\dfrac{1}{f(x)},\quad x>0.
\ee
It is straightforward to verify that
\be\label{soleqfu}
F(x)=[w]_q^\frac{k(v-u)}{2}\frac{\prod_{j=1}^v G_{q^w}^k\left(\frac{x+b_j}{w}\right)}{\prod_{i=1}^u G_{q^w}^k\left(\frac{x+a_i}{w}\right)}
\ee
is a solution of \eqref{eqfu}. In the next lemma we establish some of the functional properties of $F$.

\begin{lem}\label{Fdecder}
Let $F$ be as in \eqref{soleqfu}, and set
\[
\Upsilon_w=\sum_{i=1}^u q^{wa_i}-\sum_{j=1}^v q^{wb_j}.
\]
Then $F$ is increasing (resp. decreasing) if and only if $k\Upsilon_w >0$ (resp. $k\Upsilon_w<0$). Furthermore, for $n\geq 1$, $\dfrac{d^n}{dx^n}\log F(x)$ has the same sign as  $-k\Upsilon_w \log^n q$ for all sufficiently large $x$.
\end{lem}

\begin{proof}
First we assume $k=1$. Using the notation of Proposition \ref{quotientdecrease} and the definition of $G_q$ in \eqref{ZM1}, a straightforward computation shows that $\Upsilon$ corresponding to $F$ is given by
\[
\Upsilon=(q^w-1)\Upsilon_w,
\]
and the result follows from the third case of Proposition \ref{quotientdecrease} ($r=s=u+v$) and from Proposition \ref{quotientder}. The case for general $k$ follows from the simple observation that $\log F^k=k\log F$, and that, since $F$ is positive, $F'$ has the same sign as $\dfrac{d}{dx} \log F$.
\end{proof}

\begin{thm}\label{G3thm}
Let $f$ be a function defined for $x>0$. Assume that $f$ satisfies \eqref{eqfu}, and that either
\be\label{semidecreasing}
f(x+w) < f(x)\quad\textrm {for all } x>M\geq 0,
\ee
or
\be\label{semiincreasing}
f(x+w) > f(x)\quad\textrm {for all } x>M\geq 0,
\ee
then for all $x>0$, $f(x)=F(x)$ as in \eqref{soleqfu}.
\end{thm}

\begin{proof}
We shall prove the theorem assuming \eqref{semidecreasing}, the proof in the case of \eqref{semiincreasing} is almost identical. 

Write
\[
P(x)=\left(\frac{[x+b_1]_q\cdots[x+b_v]_q}{[x+a_1]_q\cdots [x+a_u]_q}\right)^k,
\]
and for $n\geq 0$ set $P_n:=P(x+nw)$. It follows that if $f$ is a solution of  \eqref{eqfu}, then it must satisfy
\bea
f(x+2nw)&=&\prod_{k=0}^{n-1}\dfrac{P_{2k}}{P_{2k+1}} f(x),\\
f(x+(2n+1)w)&=&\frac{1}{P_{2n}}\prod_{k=0}^{n-1}\dfrac{P_{2k+1}}{P_{2k}}\dfrac{1}{f(x)}.
\eea
Assume $(2n-1)w>M$ and $x>0$. By \eqref{semidecreasing} we have
\[ 
f(x+(2n+1)w)< f(x+2nw)< f(x+(2n-1)w).
\]
Hence
\[
\frac{1}{P_{2n}}\prod_{k=0}^{n-1}\dfrac{P_{2k+1}}{P_{2k}}\dfrac{1}{f(x)}< \prod_{k=0}^{n-1}\dfrac{P_{2k}}{P_{2k+1}} f(x)< \frac{1}{P_{2n-2}}\prod_{k=0}^{n-2}\dfrac{P_{2k+1}}{P_{2k}}\dfrac{1}{f(x)}.
\]
Since
\[
\lim_{n\to \infty}P_n=(1-q)^{(u-v)k},
\]
and $f(x)\neq 0$ for all $x>0$, 
it follows that
\be
\begin{split}
|f(x)|&=(1-q)^{k(v-u)/2} \prod_{k=0}^\infty \frac{P_{2k+1}}{P_{2k}}=F(x).
\end{split}
\ee
Now, we prove that $f(x)$ can never be negative. Suppose on the contrary that $f(x_0)<0$  for some  $x_0>0$. Then, $f(x_0+nw)$ is negative for all $n$ such that $x_0+nw>M$. Thus, for all $n\geq \left[\frac{M}{w}\right]+1$ we have
\be\label{contra}
0< f(x_0+nw)-f(x_0+(n+2)w)< -F(x_0+nw)+F(x_0+(n+2)w).
\ee
On the other hand, since both $F$ and $f$ satisfy \eqref{eqfu} we must have
\be\label{+2}
\frac{F(x_0+(n+2)w)}{F(x_0+nw)}=\frac{f(x_0+(n+2)w)}{f(x_0+nw)}
\ee
since the right hand side of \eqref{+2} is less than $1$, we see that the right hand side of \eqref{contra} is negative, leading to a contradiction.  This contradiction shows that $f(x)$ can never be negative, and the result follows.
\end{proof}

In~ \cite{Thielman}, Thielman  proved that the only  function which is convex for $x\geq K>0$ and satisfies the functional equation

\[1/f(x+w)=x^k f(x), \quad x>0,\; k>0,;w>0,\]
is 
\[f(x)=\left[\dfrac{\Gamma\left(x/2w\right)}{(2w)^{1/2}\Gamma\left((x+w)/2w\right)}\right]^k.\]
 Clearly, Mayer's result  is the particular case $a=w=1$. 
 A $q$-analogue of Thielman's result follows as a  special case of the next theorem.

\begin{thm}\label{G4thm}
Let $f$ be a function defined for $x>0$. Assume that $f$ satisfies \eqref{eqfu}, and that $f$ is either convex or concave for all $x>M\geq 0$. Then for all $x>0$, $f(x)=F(x)$ as in \eqref{soleqfu}.
\end{thm}

\begin{proof}
The proof is similar to the proof of Theorem \ref{G3thm}, where we utilize the convexity (or concavity) in place of \eqref{semidecreasing} as in the proof of Theorem \ref{Mayeranalogue}. Details are omitted for brevity. 
\end{proof}

\section{{\bf The Multiplication Formula as a Defining Property}}
Recall that the gamma function satisfies the \emph{Legendre duplication formula}

\begin{equation}\label{mult}
\Gamma(x)\Gamma\left(x+\frac{1}{2}\right)=2^{1-2x}\sqrt{\pi}\Gamma(2x),
\end{equation}
which is a special case of the \emph{Gauss multiplication formula} for integer $m\geq 2$,

\be\label{gaussmult}
\Gamma(x)\Gamma\left(x+\frac 1m\right)\cdots\Gamma\left(x+\frac{m-1}{m}\right)={m}^{\frac{1-2mx}{2}}(2\pi)^{\frac{m-1}{2}}\Gamma(mx).
\ee

For $q\in [0,1]$, a $q$-analogue of \eqref{gaussmult} is the  following functional equation (see \cite{aar} for instance)

\be\label{gammaproductformula}
\Gamma_{q^m}(x)\Gamma_{q^m}\left(x+\frac 1m\right)\cdots\Gamma_{q^m}\left(x+\frac{m-1}{m}\right)=\left([m]_q^{1-mx}\prod_{i=1}^{m-1}\Gamma_{q^m}\left(\frac{i}{m}\right)\right)\Gamma_q(mx).
\ee
Let $p=q^{-1}$. Using \eqref{gammaq-1} we see that

\be\label{pgaussmult}
\begin{split}
\prod_{i=0}^{m-1}\Gamma_{p^m}\left(x+\frac{i}{m}\right)&=\left([m]_q^{1-mx}\prod_{i=1}^{m-1}\Gamma_{q^m}\left(\frac{i}{m}\right)\right)\Gamma_q(mx)\prod_{i=0}^{m-1}p^{\frac{m}{2}(x+(\frac{i}{m}-1))(x+(\frac{i}{m}-2))}\\
&=\left([m]_p^{1-mx}\prod_{i=1}^{m-1}\Gamma_{p^m}\left(\frac{i}{m}\right)\right)\Gamma_p(mx),
\end{split}
\ee
and thus \eqref{gammaproductformula} holds for all $q>0$.

\;
In \cite[p. 38]{anasta} the following theorem was proved.
\begin{thm}\label{anasta38}
Let $\phi$ be a periodic function with period 1 that is positive on $[0,1]$, and having continuous second derivative on $[0,1]$.
If $\phi$ satisfies
\begin{equation}
\phi(x)\phi(x+1/2)=c\phi(2x),\;c>0
\end{equation}
 then $\phi$ itself must be constant on $\R$.
\end{thm}

We now state and prove a generalized  $q$-analogue of the above theorem, which we then use to show that \eqref{gammaproductformula} characterizes the $q$-gamma function.
\begin{thm}\label{constant}
Let $\phi(x;q)$ be a  continuous function defined for $x\geq 0$ and $q\in[0,1]$ such that
 \begin{enumerate}
 \item For each fixed $q$, $\phi$ is a  positive periodic function with period one.

  \item The function  $\phi(x;q)$ has continuous  partial derivatives with respect to the variable $x$ up to order 2   on       $[0,1]\times [0,1]$.
      \end{enumerate}
 Let $m\geq 2$ be an integer, and assume that $\phi$ satisfies
\begin{equation}
\phi(x;q)\phi(x+1/m;q)\cdots \phi(x+(m-1)/m)=c(q)\phi(mx;q^{1/m}),\;c(q)>0,\; x\geq 0,
\end{equation}
then, for each fixed $q\in[0,1]$,  $\phi$ itself must be constant for all $x\in \R$.
\end{thm}
\begin{proof}
Set $g(x;q)=\dfrac{\partial^2}{\partial^2x}\log\phi(x;q)$. Then $g(x;q)$ satisfies the functional equation
\begin{equation}\label{eq:AZ2}
\dfrac{1}{m^2}\left[g(x;q)+g(x+1/m;q)+\dots+g(x+(m-1)/m)\right]=g(mx;q^{1/m}),\;x\geq 0.
\end{equation}
Now $g(x;q)$ is continuous for $(x,q)\in[0,1]\times[0,1]$. Then it is bounded there, i.e. there exists a positive constant $K>0$ such that
\begin{equation}
|g(x;q)|\leq K , \quad\mbox{for all}\quad (x,q)\in[0,1]\times [0,1].
\end{equation}
Since $g$, as a function of $x$, is periodic with period 1, then
\begin{equation}
|g(x;q)|\leq K,\quad \mbox{for all}\quad (x,q)\in[0,\infty)\times [0,1].
\end{equation}
Substituting in \eqref{eq:AZ2} gives
\begin{equation}\label{eq:Az3}
|g(mx;q^{1/m})|\leq \frac{1}{m}K\quad\mbox{ for all}\; x\in[0,\infty)\;\mbox{and}\;q\in\,[0,1].
\end{equation}
That is
\begin{equation}\label{eq:AZ4}
|g(x;q)|\leq \frac{K}{m},\quad \mbox{for all}\quad (x,q)\in[0,\infty)\times [0,1].
\end{equation}
 Substituting again from \eqref{eq:AZ4} in \eqref{eq:AZ2} gives
 \begin{equation}
|g(x;q)|\leq \frac{K}{m^2},\quad \mbox{for all}\quad (x,q)\in[0,\infty)\times [0,1].
\end{equation}
Continuing in this process yields that
\begin{equation}
|g(x;q)|\leq \frac{K}{m^n},\quad \mbox{for all}\quad (x,q)\in[0,\infty)\times [0,1],
\end{equation}
and for all $n\in\mathbb{N}$. Consequently $g(x;q)$ is identically zero in $[0,\infty)$. This implies that
 \[\phi(x;q)=e^{\alpha x+\beta},\]
 for some constants $\alpha$ and $\beta$ (that may depend on $q$ and $m$). Since $\phi$ is periodic function of period one, then $\alpha=0$. Hence $\phi $ is a constant function and the theorem follows.
 \end{proof}

\begin{thm}\label{thm2}
Let $f:(0,\infty)\times[0,1]$ be positive continuous  function that satisfies
\begin{itemize}
\item[(i)] $f(x+1;q)=[x]_qf(x;q)$,
\item[(ii)]$f(1;q)=1$,
\item[(iii)] the partial derivatives up to order \textup{2} with respect to the variable $x$ exist and continuous on for $x\in(0,\infty)$ and  $q\in [0,1]$.
\end{itemize}
If, for some integer $m\geq 2$ and some positive constant (in $x$) $\alpha_m(q)$, $f$ satisfies 
\be\label{qgammamult}
f(x;q^m)f\left(x+\frac 1m; q^m\right)\cdots  f\left(x+\frac{m-1}{m};q^m\right)=\alpha_m(q)[m]_q^{1-mx}f(mx;q),
\ee
then $f(x;q)=\Gamma_q(x)$, and hence satisfies \eqref{qgammamult} for all $m$ with
\[
\alpha_m(q)=\prod_{i=1}^{m-1}\Gamma_{q^m}\left(\frac{i}{m}\right).
\]
\end{thm}

\begin{proof}
Let $n\in\mathbb{N}$ be arbitrary. From \eqref{qgammamult},  we obtain
\begin{equation}\label{eq:AZ1}
f(x+n;q^m)f(x+n+\frac{1}{m};q^m)\cdots f(x+n+\frac{m-1}{m};q^m)=\alpha_m(q) [m]_q^{1-mx-mn}f(mx+mn;q).
\end{equation}
From (i), we have
\begin{equation}\label{pf:Thm 2.6}
f(x+n;q)=\frac{(q^{x};q)_n}{(1-q)^{n}} f(x;q), \quad x>0.
\end{equation}
Substituting from \eqref{pf:Thm 2.6} in \eqref{eq:AZ1} yields
\begin{equation}\label{technical0}
(1-q^m)^{-mn}\prod_{i=0}^{m-1}(q^{mx+i};q^m)_n f(x+\frac{i}{m};q^m)=\alpha_m(q)[m]_q^{1-mx-mn}(1-q)^{-mn}(q^{mx};q)_{mn}f(mx;q),  
\end{equation}
which simplifies to
\begin{equation}\label{technical1}
\prod_{i=0}^{m-1}(q^{mx+i};q^m)_n f(x+\frac{i}{m};q^m)=\alpha_m(q)[m]_q^{1-mx}(q^{mx};q)_{mn}f(mx;q). 
\end{equation}
Now for $x\in [0,\infty)$ set
\[
\phi(x;q)=\begin{cases}
\dfrac{f(x;q)}{\Gamma_q(x)}\, &\textrm{ if } x>0, \\
 1 \, &\textrm{ if } x=0.
\end{cases}
\]
Notice that conditions (i) and (ii) imply that
\[
\lim_{x\to 0^+} \frac{1-q^x}{1-q} f(x;q)=1.
\]
Thus
\be\label{philim}
\lim_{x\to 0^+} \phi(x;q)=1,
\ee
\eqref{gammalim} we deduce that $\phi(x;q)$ is continuous on $[0,\infty)\times[0,1]$.
Multiplying both sides of \eqref{technical1} by 
\[
\prod_{i=0}^{m-1}\frac{(1-q^m)^{x+\frac{i}{m}-1}}{(q^m;q^m)_\infty}=\dfrac{(1-q^m)^{mx-(m+1)/2}}{(q^m;q^m)_{\infty}^m}
\] and calculating the limit as $n$ tends to infinity give
\begin{equation}
\phi(x;q^m)\phi(x+\frac{1}{m};q^m)\cdots\phi(x+\frac{m-1}{m};q^m)=c(q)\phi(mx;q),\quad x\geq 0,
\end{equation}
where
\begin{equation}
c(q)=\dfrac{\alpha_m(q)(1-q)^{(1-m)/2}(q;q)_\infty}{(q^m;q^m)^m_{\infty}}.
\end{equation}
Applying Theorem \ref{constant}  yields that  $\phi(x;q)$ is identically constant on $[0,\infty)$. Since $\phi(1;q)=1$, then $\phi(x;q)\equiv 1$ on $[0,\infty)$. This proves the theorem.
\end{proof}

\begin{thm}\label{thm4}
Let $F:(0,\infty)\times[1,\infty)$ be positive continuous  function that satisfies
\begin{itemize}
\item[(i)] $F(x+1;p)=[x]_pF(x;p)$,
\item[(ii)] $F(1;p)=1$,
\item[(iii)] The partial derivatives up to order \textup{2} with respect to the variable $x$ exist and are continuous for $x\in(0,\infty)$
\item[(iv)] For $0<x,y\leq 1$, the limit
\[
h(x):=\lim_{(y,p) \to (x,\infty)}F(y,p)p^\frac{-(y-1)(y-2)}{2}
\]
is a well-defined twice continuously differentiable function.
\end{itemize}
Assume that for some integer $m\geq 2$ there exists a positive constant $\alpha_m(p)$ such that 
\be\label{pmult}
F(x;p^m)F\left(x+\frac 1m; p^m\right)\cdots  F\left(x+\frac{m-1}{m};p^m\right)=\alpha_m(p)[m]_p^{1-mx}F(mx;p),
\ee
then $F(x;p)=\Gamma_p(x)$, and hence satisfies \eqref{pmult} for all $m$ with
\[
\alpha_m(p)=\prod_{i=1}^{m-1}\Gamma_{p^m}\left(\frac{i}{m}\right).
\]
\end{thm}
\begin{proof}
For $y>0$, write 
\[
h_p(y)=F(y,p)p^\frac{-(y-1)(y-2)}{2}.
\]
Let $x>0$ be such that $\displaystyle \lim_{(y,p)\to (x,\infty)} h_p(y)$ exists, and denote that limit by $h(x)$. Then, using (i) we see that
\be\label{hperiod}
\lim_{(y,p)\to (x,\infty)} h_p(y+1)=\lim_{p\to \infty}[y]_p p^{1-y}h_p(y)=h(x),
\ee
where the last equality follows since
\be\label{[x]p}
\lim_{(y,p)\to (x,\infty)} [y]_pp^{1-y}=\lim_{p\to \infty} \frac{p-p^{1-x}}{p-1}=1.
\ee

It follows from \eqref{hperiod} and (iv) that for all $x>0$ 
\begin{equation}\label{hp}
h(x):=\lim_{(y,p)\to (x,\infty)}h_p(x)
\end{equation}
 is a well-defined twice continuously differentiable function that is periodic with period $1$. 

Next, for $x> 0$ and $q\in[0,1]$, we set
\[
f(x;q):=
\begin{cases}
F(x;q^{-1})q^\frac{(x-1)(x-2)}{2} &,\textrm{ if } 0<q\leq1,\\
h(x) &,\textrm { if } q=0.
\end{cases}
\]
It follows from (iii) and \eqref{hp} that $f(x;q)$ is continuous on $(0,\infty)\times [0,1]$. A straightforward computation as in \eqref{pgaussmult} shows that the functional equation \eqref{pmult} for $F$ translates into \eqref{qgammamult}; thus the conditions of Theorem \ref{thm2} are satisfied for $f$, and the result follows.
\end{proof}
\begin{rem}
In~\cite{K-M}, the authors show that  the $q$-gamma function satisfies
\[
\Gamma_q(x)\Gamma_q(x+1/2)=Q(x)\Gamma_q(2x),
\]
where
\[
Q(x):=\Gamma_q(1/2)\prod_{n=0}^{\infty}\dfrac{(1-q^{n+2x})(1-q^{n+1/2})}{(1-q^{n+x})(1-q^{n+x+1/2})}.
\]
 which is different from \eqref{gammaproductformula}. They also introduced a uniqueness theorem for the solution of the functional equation
\[
f(x)f(x+1/2)=Q(x)f(2x),\quad f(1)=1,
\]
by using a technique different from the one we used above.
\end{rem}

\section{{\bf Functional Properties of the $q$-beta function}}

The $q$-analogue of the beta function is defined  for $q>0$ by

\be\label{qbeta-qgamma relation}
B_q(x,y)=\dfrac{\Gamma_q(x)\Gamma_q(y)}{\Gamma_q(x+y)},\;x>0,\;y>0.
\ee
See~\cite[p. 22]{GR}.  Hence we have the symmetry relation
\[B_q(x,y)=B_q(y,x),\quad x>0,\;y>0. \]
Using \eqref{gammaq-1} we can prove that
\be \label{qbeta-1}B_q(x,y)=q^{1-xy}B_{q^{-1}}(x,y),\;q>0.\ee

\begin{lem}\label{beta decrease}
For each fixed $y>0$ and $q>0$,  $B_q(x,y)$ is  a decreasing function for $x>0$.
\end{lem}
\begin{proof}
For $q=1$ this is a well-known classical fact. For $0<q<1$, the result follows by applying Proposition \ref{quotientdecrease}, with $r=s=1$, and $a_1=0$, $b_1=y$. For $q>1$, note that $q^{1-xy}$ is decreasing in $x$, and thus the result follows from the case $q\in (0,1)$ we just proved and \eqref{qbeta-1}.
\end{proof}

\begin{lem}
 For each fixed $y>0$, $B_q(x,y)$, $q>0$, is logarithmically convex for $x>0$.
\end{lem}
\begin{proof}
For $q=1$ this is a well-known classical fact. For $0<q<1$, the result follows by applying Proposition \ref{quotientder}, with $r=s=1$, $n=2$,  and $a_1=0$, $b_1=y$.  This also proves the case  $q>1$ since, by\eqref{qbeta-1} we have
\[
\frac{d^2}{dx^2}\log B_q(x,y)=\frac{d^2}{dx^2}\log B_{q^{-1}}(x,y).
\]
\end{proof}
It is easy to see that
\be\label{qbeta}
B_q(x+1,y)=\frac{1-q^x}{1-q^{x+y}}B_q(x,y).
\ee
The following two theorems show that \eqref{qbeta}, together with either  logarithmic convexity or being decreasing  (in one of the variables) and an initial value, uniquely determine the $q$-beta function. 
\begin{thm}\label{betathmlog}
If $f(x)$ is a positive function for $x>0$ and for some $y>0$ and some positive $q\neq 1$ we have
\begin{itemize}
\item[(i)] $f(x+1)=\frac{1-q^x}{1-q^{x+y}} f(x)$,
\item[(ii)] $f(x)$ is logarithmically convex  for $x>M\geq 0$, and
\item[(iii)] $f(1)=\frac{1-q}{1-q^y}$,
\end{itemize}
then $f(x)=B_q(x,y)$ for all $x>0$.
\end{thm}

\begin{proof}
The result follows by applying Theorem \ref{logconvexwab} and Corollary \ref{moakgeneral} to $\frac{f(x)}{f(1)}$, with $k=w=r=s=1$ and $a_1=0$, $b_1=y$.
\end{proof}

\begin{thm}\label{thma}
If  $f(x)$ is a positive function for $x>0$ and  for some $y>0$ and some  positive $q\neq 1$ we have
\begin{itemize}
\item[(i)] $f(x+1)=\frac{1-q^x}{1-q^{x+y}} f(x)$,
\item[(ii)] $f(x)$ is decreasing for $x>M\geq 0$,
\item[(iii)] $f(1)=\frac{1-q}{1-q^y}$,
\end{itemize}
then $f(x)=B_q(x,y)$ for all $x>0$.
\end{thm}

\begin{proof}
As in the preceding proof, the proof follows from Theorem \ref{KMgeneral} and Corollary \ref{KMgeneralwq1}.
\end{proof}

We also have the following result which is stronger than Theorem \ref{thma} when $q>1$.
\begin{cor}\label{thmb}
If  $f(x)$ is a positive function for $x>0$ and for some $y>0$ and some  positive $q\neq 1$ we have
\begin{itemize}
\item[(i)] $f(x+1)=\frac{1-q^x}{1-q^{x+y}} f(x)$,
\item[(ii)] $q^{xy}f(x)$ is decreasing for $x>M\geq 0$,
\item[(iii)] $f(1)=\frac{1-q}{1-q^y}$,
\end{itemize}
then $f(x)=B_q(x,y)$ for all $x>0$.
\end{cor}

\begin{proof}
Take $h(x)=q^{xy-1}f(x)$. It follows from \eqref{qbeta-1} that $h$ satisfies the conditions of Theorem \ref{thma} where  the $q$ in the theorem is replaced by $q^{-1}$.
Hence, $h(x)=B_{q^{-1}}(x,y)$  and 
\[
f(x)=q^{1-xy}B_{q^{-1}}(x,y)=B_q(x,y).
\]

\end{proof}

 In all of the previous theorems of this section we considered the variable $y$ as a fixed positive number.  We now state a defining property of $q$-Beta in both variables. First, note that by \eqref{qbeta-qgamma relation} we have
\[
B_q(x+1,y+1)=\dfrac{(1-q^x)(1-q^y)}{(1-q^{x+y})(1-q^{x+y+1})} B_q(x,y).
\]
Also, applying Lemma \ref{beta decrease} twice we see that if $x>u$ and $y>v$ then
\[
B_q(x,y)<B_q(u,v).
\]

\begin{thm}\label{thmc}
Assume  that for some $0<q<1$,  $f(x,y)$   $(x,y>0)$ satisfies
\begin{itemize}
\item[(i)] $f(x+1,y+1)=\dfrac{(1-q^x)(1-q^y)}{(1-q^{x+y})(1-q^{x+y+1})} f(x,y)$,
\item[(ii)] $f(x,y)<f(u,v)$ whenever $x>u>M$ and $y>v>M$, for some $M\geq 0$,
\item[(iii)] $f(1,1)=1$,
\end{itemize}
then $f(x,y)=B_q(x,y)$.

\end{thm}
\begin{proof}
Using (i) it suffices to prove the result for $(x,y)\in(0,1]\times (0,1]$.  Let $(x,y)\in(0,1]\times (0,1]$ and $n>[M]$. From (i), we see that for any positive integer $m$ we have
\begin{equation}\label{im}
f(m+u,m+v)=\dfrac{(q^u;q)_m(q^{v};q)_m}{(q^{u+v};q)_{2m}}f(u,v).
\end{equation}
Set $h(x,y)=\frac{f(x,y)}{(q^{x+y};q)_\infty}$.  From (ii) and the fact that if $(q^x;q)_\infty$ is increasing for $x>0$ we have
\begin{equation}\label{pf2:1}
h(n,n)\geq h(x+n,y+n) \geq h(n+1,n+1).
\end{equation}
We use  \eqref{im} with
$u=v=1, m=n-1$, $u=v=x,m=n$, and $u=v=1,m=n$  and substitute in \eqref{pf2:1} to get
\[
\dfrac{(q;q)^2_{n-1}}{(q^{2};q)_{2n-2}(q^{2n};q)_\infty}f(1,1) \geq \dfrac{(q^x;q)_n(q^{y};q)_n}{(q^{x+y};q)_{2n}(q^{2n+x+y};q)_\infty}f(x,y)\geq \dfrac{(q;q)_{n}^2}{(q^{2};q)_{2n}(q^{2n+2};q)_\infty}f(1,1),
\]
for all $n>[M]$. As $f(1,1)=1$ we get
\[
\dfrac{(q^{x+y};q)_{2n} (q^{2n+x+y};q)_\infty (q;q)^2_{n-1}}{(q^{x};q)_n(q^y;q)_n(q^{2};q)_{2n-2}(q^{2n};q)_\infty} \geq f(x,y)\geq \dfrac{(q^{x+y};q)_{2n}(q^{2n+x+y};q)_\infty(q;q)^2_{n}}{(q^{2};q)_{2n}(q^{2n+2};q)_\infty(q^{x};q)_n(q^y;q)_n}.
\]
Consequently
\begin{equation}
\begin{split}
f(x,y)&=\lim_{n\to\infty} \dfrac{(q^{x+y};q)_{2n} (q^{2n+x+y};q)_\infty (q;q)^2_{n-1}}{(q^{x};q)_n(q^y;q)_n(q^{2};q)_{2n-2}(q^{2n};q)_\infty}\\
&=\dfrac{(q^{x+y};q)_\infty}{(q^{x};q)_\infty(q^y;q)_{\infty}}(q;q)_{\infty}(1-q).
\end{split}
\end{equation}
That is  $f(x,y)=B_q(x,y)$ and the theorem follows.
\end{proof}

\begin{thm}
Let $q>1$ and assume that   $f(x,y)$   $(x,y>0)$ satisfies
\begin{itemize}
\item[(i)] $f(x+1,y+1)=\dfrac{(1-q^x)(1-q^y)}{(1-q^{x+y})(1-q^{x+y+1})} f(x,y)$,
\item[(ii)] $q^{xy}f(x,y)<q^{uv}f(u,v)$ whenever $x>u>M$ and $y>v>M$, for some $M>0$,
\item[(iii)] $f(1,1)=1$,
\end{itemize}
then $f(x,y)=B_q(x,y)$.
\end{thm}

\begin{proof}
The proof follows by verifying that the function $h(x,y):=q^{xy}f(x,y)$ satisfies the condition of Theorem \ref{thmc} with $q$ is replaced by $q^{-1}$. Hence
\[h(x,y)=B_{q^{-1}}(x,y),\]
which is equivalent to $f(x,y)=B_q(x,y)$.
\end{proof}


\section{\bf An approximate $q$-analogue of Euler's reflection formula}
Recall the following famous formula of Euler
\begin{equation}\label{euler}
\Gamma(x)\Gamma(1-x)=\frac{\pi}{\sin(\pi x)}.
\end{equation}
One proof of \eqref{euler} relies on Theorem \ref{anasta38}. Thus it is natural to try to get a $q$-reflection formula using Theorem \ref{constant}. With this in mind, we  define a function $\phi(x;q)$ by
\[
\phi(x;q):=\begin{cases} \Gamma_q(x)\Gamma_q(1-x)\sin(\pi x)q^{\frac{x(x-1)}{2}} &\textrm{if } x\notin \Z,\,q \in [0,1],\\
\frac{\pi (1-q)}{\log(q^{-1})} &\textrm{if } x\in \Z,\, 0<q<1,\\
\pi &\textrm{if } x\in \Z,\,q=1,\\
0 &\textrm{if } q=0. \\
\end{cases}
\]
Note that $\phi(x;q)$ is continuous for $(x,q)\in \R\times [0,1]$. Furthermore, for $q \in (0,1]$ arbitrary but fixed, $\phi(x;q)$ is a positive,  periodic, twice continuously differentiable function (of $x$) with period $1$.
Now using the $q$-analogue of the Legendre duplication formula, namely
\[
\Gamma_{q^2}(x)\Gamma_{q^2}\left(x+\frac{1}{2}\right)=\frac{\Gamma_{q}(2x)\Gamma_{q^2}(\frac{1}{2})}{(1+q)^{2x-1}},
\]
we see that for $q\in (0,1]$, $\phi(x;q)$ satisfies the functional equation
\[
\phi(x;q)\phi\left(x+\frac{1}{2};q\right)=c(q)\phi(2x;q^\frac{1}{2}),
\]
with
\[
c(q)=\frac{(1+q^\frac{1}{2})\Gamma_q^2(\frac{1}{2})}{2q^\frac{1}{8}}.
\]
Thus the conditions of Theorem \ref{constant} are ``almost" satisfied; they are satisfied for $q \in (0,1]$ rather than $q\in [0,1]$. This of course makes a big difference and the conclusion of this theorem is in fact not true. However we could hope that  a weaker version holds and that $\phi(x;q)$ is ``approximately constant" (in $x$). If that's the case, and taking that ``approximate constant" to be
\[\phi\left(\frac{1}{2};q\right)=\Gamma_q^2\left(\frac{1}{2}\right) q^\frac{1}{8},
\]
then we suspect that
\[
\sin(\pi x)\approx \frac{\Gamma_q(\frac{1}{2})q^{\frac{1}{2}(x-\frac{1}{2})^2}}{\Gamma_q(x)\Gamma_q(1-x)}.
\]

Indeed, in \cite{Askey78} Askey proves the following formula (using Jacobi's triple product identity and the Poisson summation formula)
\begin{equation}\label{AskeyReflection}
\frac{\Gamma_q^2(\frac{1}{2})}{\Gamma_q(\frac{1}{2}+x)\Gamma_q(\frac{1}{2}-x)}=q^{-\frac{x^2}{2}}\frac{\cos(\pi x)+h(r,x)}{1+h(r,0)},
\end{equation}
where the auxiliary variable $r$ is defined by $r=r(q):=e^{\frac{\pi^2}{2\log q}}$ (or equivalently $\log r \log q=\frac{\pi^2}{2}$) and
\[
h(r,x)=\sum_{n=1}^\infty r^{n(n+1)}\cos((2n+1)\pi x).
\]
 Replacing $x$ by $x-1/2$ in \eqref{AskeyReflection} gives

\[
\frac{\Gamma_q^2(\frac{1}{2})}{\Gamma_q(x)\Gamma_q(1-x)}=\frac{q^\frac{-(x-\frac{1}{2})^2}{2}h(r,x-1/2)}{1+h(r,0)}.
\]
It is easy to see that
\[
h(r,x-1/2)=\sin(\pi x)h_c(r,x)+\cos(\pi x)h_s(r,x),
\]
where
\begin{equation}
\begin{split}
h_c(r,x)&=\sum_{n=1}^\infty (-1)^n r^{n(n+1)}\cos(2\pi n x), \\ h_s(r,x)&=\sum_{n=1}^\infty (-1)^n r^{n(n+1)}\sin(2\pi n x).
\end{split}
\end{equation}
It follows that
\begin{equation}\label{reflectionsine}
\frac{\Gamma_q^2(\frac{1}{2})}{\Gamma_q(x)\Gamma_q(1-x)}=q^\frac{-\left(x-\frac{1}{2}\right)^2}{2}\frac{\sin(\pi x)(1+h_c(r,x))+\cos(\pi x)h_s(r,x)}{1+h(r,0)} ,
\end{equation}

\begin{rem}
Let $p=1/q$. Since $\Gamma_p(x)=\Gamma_q(x)p^{\frac{(x-1)(x-2)}{2}}$, it is straightforward to see that
\begin{equation}\label{qpinvariance}
\frac{\Gamma_p^2(\frac 12) p^{\frac{(x-\frac 12)^2}{2}}}{\Gamma_p(x)\Gamma_p(1-x)}= \frac{\Gamma_q^2(\frac 12) q^{\frac{(x-\frac 12)^2}{2}}}{\Gamma_q(x)\Gamma_q(1-x)},
\end{equation}
and \eqref{reflectionsine} remains true when $q$ is replaced by $p$. Thus all the results in this section have their straightforward counterparts with $p$ in place of $q$. for brevity we shall only state the results for $0<q<1$.
\end{rem}

As a consequence of \eqref{reflectionsine} we get the following surprising approximation for $\pi$.

\begin{thm}\label{PiTheorem}
For all $0<q<1$ we have

\begin{equation}\label{PiFormula}
\pi= q^\frac{1}{8}\cdot\Gamma_q^2\left(\frac{1}{2}\right)\cdot\dfrac{ \log(q)}{q-1}\cdot \frac{1+h_c(r,0)}{1+h(r,0)}.
\end{equation}
Hence
\begin{equation}\label{PiApprox}
\pi= q^\frac{1}{8}\cdot\Gamma_q^2\left(\frac{1}{2}\right)\cdot\dfrac{ \log(q)}{q-1}+O(e^\frac{\pi^2}{\log q}) \, \textrm{ as } q\to 1^-.
\end{equation}
\end{thm}

\begin{proof}
Note that $r\in (0,1)$ for $q\in(0,1)$. Furthermore,
$h_c(r,x)$ and $h_s(r,x)$ are continuous functions of $x$ for all $\abs{r}<1$ and  $h_s(r,0)=0$. Thus, taking the limit in \eqref{reflectionsine} as $x\to 0$ gives
\[
\begin{split}
q^\frac{1}{8} \Gamma_q^2\left(\frac{1}{2}\right)\frac{1+h(r,0)}{1+h_c(r,0)}&=\lim_{x \rightarrow 0}\sin(\pi x)\Gamma_q(x)\\
&=(q-1) \frac{\pi}{\log q},
\end{split}
\]
and \eqref{PiFormula} follows immediately.

Next, note that
\begin{equation}\label{hbound}
h(r,0)=\sum_{n=1}^\infty r^{n(n+1)} < \sum_{n=1}^\infty r^{2n}=\frac{r^2}{1-r^2}.
\end{equation}
Hence
\[
1+r^2\leq1+ h(r,0)\leq 1+\frac{r^2}{1-r^2}.
\]
We also trivially have
\[
1-r^2\leq 1+h_c(r,0)\leq 1.
\]
Thus
\[
(1-r^2)^2\leq \frac{1+h_c(r,0)}{1+h(r,0)}\leq \frac{1}{1+r^2}.
\]
It follows that
\begin{equation}\label{TightError}
1-2r^2+r^4\leq \frac{1+h_c(r,0)}{1+h(r,0)}\leq 1-r^2+r^4,
\end{equation}
and hence $\frac{1+h_c(r,0)}{1+h(r,0)}=1+O(r^2)$ as $r\to 0$. Equation \eqref{PiApprox} follows since $r^2=e^{\frac{\pi^2}{\log q}}$.
\end{proof}
Using a similar analysis, we obtain the following two $q$-approximations of $\sin(\pi x)$.

\begin{thm}\label{sintheorem}
For $4\times 10^{-13}<q<1$ and all $x\in \R$ we have the following approximate formulas as $q\to 1^-$

\begin{enumerate}
\item[(i)]
\begin{equation}\label{sin1}
\begin{split}
\sin(\pi x)&=q^{\frac{1}{2}(x-\frac{1}{2})^2} \frac{\Gamma_q^2(\frac{1}{2})}{\Gamma_q(x)\Gamma_q(1-x)}+O(e^\frac{\pi^2}{\log q})\\
&=q^{\frac{1}{2}(x-\frac{1}{2})^2}\prod_{k=0}^\infty \frac{(1-q^{k+x})(1-q^{k+1-x})}{(1-q^{k+\frac{1}{2}})^2}+O(e^\frac{\pi^2}{\log q}).
\end{split}
\end{equation}
\item[(ii)]
\begin{equation}\label{sin2}
\begin{split}
\sin(\pi x)&=\pi q^{\frac{x(x-1)}{2}}\frac{(1-q)}{\log(q^{-1})\Gamma_q(x)\Gamma_q(1-x))}+O(e^\frac{\pi^2}{\log q})\\
&=\pi q^{\frac{x(x-1)}{2}} \frac{1}{\log(q^{-1})}\prod_{k=0}^\infty \frac{(1-q^{k+x})(1-q^{k+1-x})}{(1-q^{k+1})^2}+O(e^\frac{\pi^2}{\log q}).
\end{split}
\end{equation}
\end{enumerate}

\end{thm}

\begin{proof}
It is obvious that, for all $x\in \R$, each of $h(r,x)$, $h_c(r,x)$ and $h_s(r,x)\cos(\pi x)$ is bounded below by $-h(r,0)$ and bounded above by $h(r,0)$. It then follows from \eqref{hbound} that
\begin{equation}
\frac{1-2r^2}{1-r^2}<1-h(r,0)\leq 1+h(r,x) \leq1+h(r,0)<\frac{1}{1-r^2}.
\end{equation}
Thus
\[
1-2r^2 \leq\frac{1+h_c(r,x)}{1+h(r,x)} \leq \frac{1}{1-2r^2}=1+2r^2+\dots+(2r^2)^n+\dots.
\]
Note that the series on the right hand side converges if and only if $r<\frac{1}{\sqrt{2}}$. Using the relation $r(q):=e^{\frac{\pi^2}{2\log q}}$, we see that this will be satsified whenever  $4\times 10^{-13}<q<1$. It follows that for such $q$ we have

\begin{equation}\label{hcbound}
\frac{1+h_c(r,x)}{1+h(r,x)}=1+O(r^2)\, \textrm{ as } r\to 0.
\end{equation}
In a similar way we get
\[
-r^2<\frac{-h(r,0)}{1+h(r,0)} \leq \frac{\cos(\pi x)h_s(r,x)}{1+h(r,x)}\leq \frac{h(r,0)}{1-h(r,0)}<\frac{r^2}{1-2r^2}.
\]
Hence
\begin{equation}\label{hsbound}
\cos(\pi x)\frac{h_s(r,x)}{1+h(r,x)}=1+O(r^2) \, \textrm{ as } r\to 0.
\end{equation}
Part (i) now follows from \eqref{reflectionsine}, \eqref{hcbound}, and \eqref{hsbound}. Part (ii) follows from part (i) and \eqref{PiApprox}.

\end{proof}

\begin{rem}\label{reviewer}
We are grateful to Professor Mourad Ismail and the anonymous referee for pointing out that,  in \cite{Gosper2001}, Gosper introduced a $q$-analogue $\sin_q$ of the sine function, defined for $0<q<1$ by
\begin{equation}\label{qsine}
\sin_q(\pi z)=q^{\frac{(z-1)^2}{4}}\frac{(q^{2z};q^2)_\infty (q^{2-2z};q^2)_\infty}{(q;q^2)_\infty^2}.
\end{equation}
For this function, it is easy to obtain the following $q$-reflection formula
\begin{equation}\label{sinqref}
\sin_q(\pi z)=q^{\frac{1}{4}}\Gamma_{q^2}^2 \left(\frac{1}{2}\right)\frac{q^{z(z-1)}}{\Gamma_{q^2}(z)\Gamma_{q^2}(1-z)},
\end{equation}
which has the advantage of being exact, whereas \eqref{sin1} has the advantage of relating $q$-gamma to the classical sine. Furthermore, the referee informed us that  Mez\H{o} in \cite{Mezo2012}  recently gave a rigorous proof that the $q$-sine functions satisfies the following duplication formula, which was discovered computationally by Gosper in \cite{Gosper2001}
\[
\sin_q(2z)=q^{-\frac{1}{4}}\frac{(q^2;q^4)_\infty^4}{(q;q^2)_\infty^2} \sin_{q^2}(z)\cos_{q^2}(z).
\]
This strengthens the analogy between $q$-gamma and $q$-sine and their classical counterparts.
\end{rem}

We can also  view \eqref{sin2} as a two parameter approximation of $\pi$ which generalizes Theorem \ref{PiTheorem}.

\begin{cor}\label{picor}
For $4\times 10^{-13}<q<1$ and all $x\in \R$  we have
\begin{equation}\label{pixq}
\begin{split}
\pi&=\sin(\pi x) q^{-\frac{x(x-1)}{2}}\frac{\log(q^{-1})\Gamma_q(x)\Gamma_q(1-x))}{(1-q)}+O(e^\frac{\pi^2}{\log q})\\
&=\sin(\pi x) q^{\frac{-x(x-1)}{2}} \log(q^{-1})\prod_{k=0}^\infty \frac{(1-q^{k+1})^2}{(1-q^{k+x})(1-q^{k+1-x})}+O(e^\frac{\pi^2}{\log q}) \, \textrm{ as } q\to 1^-.
\end{split}
\end{equation}
\end{cor}

\begin{rem}
The lower bound on $q$ in Theorem \ref{sintheorem} and Corollary \ref{picor} could be made smaller if we use sharper bounds on $h(r,x),\, h_c(r,x)$, and $h_s(r,x)$. This could be attained by simply using more terms in the series for the approximation of $h(r,0)$ in \eqref{hbound} and substituting accordingly in \eqref{TightError}, \eqref{hcbound}, and \eqref{hsbound}.
\end{rem}

\begin{rem}
We note that $r$ is a rapidly decreasing function of $q$. For instance, if $q\in [0.01, 0.99]$, then $r\in [0.275\dots \times 10^{-426},0.117\dots]$. So in practice the above formulas give rather accurate approximations of $\sin(\pi x)$ and $\pi$, especially for $q$ not close $0$. On the other hand, it should be noted that as $q$ gets too close to $1^-$, the convergence of the infinite products defining $\Gamma_q(x)$ becomes slower.
\end{rem}

\begin{rem}
It would be interesting to see whether the heuristics in the beginning of this section could be made quantitative and formal. In other words, is it possible to modify Theorem \ref{constant} to give an alternative proof of Askey's reflection formula \eqref{AskeyReflection}?
\end{rem}

\section{The Asymptotic expansion as a defining property}

It is well known that the classical gamma function has the following asymptotic expansion (also known as Stirling's formula)
\begin{equation}\label{stirling}
\log \Gamma(x)\sim (x-1/2)\log x -x+\log\sqrt{2\pi} \, \textrm{ as } x\to \infty.
\end{equation}
The function 
\[
\mu(x):=\log \Gamma(x)- (x-1/2)\log x +x-\log\sqrt{2\pi}, 
\]
which was introduced by Plana \cite{Plana1820}, can be seen to be a decreasing function. If we set \[
g(x):=\left(x+\frac 12\right)\log\left(1+\frac1x\right)-1,
\]
 then $\mu(x)$ satisfies the functional equation
\begin{equation}\label{mufunceq}
\mu(x+1)=\mu(x)-g(x).
\end{equation}
It is easy to see that for $x>0$ the series
\[
\sum_{n=0}^\infty g(x+n)
\]
is convergent. In \cite{anasta} it is shown that if for $x>0$, $\nu(x)$ is a decreasing function satisfying 
\[
\nu(x+1)=\nu(x)-g(x),
\]
then up to an additive constant $k$ (which is determined by $\nu(1)$ for instance) we will have
\[
\nu(x)=k+\sum_{n=0}^\infty g(x+n).
\]
It follows that if $f(x)$ is a function satisfying
\begin{itemize}
\item[(i)] $f(x+1)=xf(x)$,
\item[(ii)] $\mu_f(x):=\log f(x)-(x-1/2)\log x+x-\log(\sqrt{2\pi})$ is decreasing for $x>0,$
\item[(iii)] $f(1)=1$,
\end{itemize}
then we must have $f(x)=\Gamma(x)$. (Since, by condition (i), $\mu_f$ satisfies the same functional equation as \eqref{mufunceq} and has the same value at $1$ as $\mu$, thus by the result of  Anastassiadis, $\mu$ and $\mu_f$ are equal, and hence $f$ and $\Gamma$ must be equal.) In other words, the main term of the asymptotic expansion for $\Gamma$ together with its factorial property completely define it.

In \cite{Moak84}, Moak derived the following $q$-analogue of Stirling's formula.

\begin{equation}\label{qStirling}
\log \Gamma_q(x)\sim (x-1/2)\log\left(\frac{1-q^x}{1-q}\right)+\frac{1}{\log q} \int_0^{-x\log q} \frac{u}{e^u-1}\, du+M_q,
\end{equation}
where
\begin{equation}\label{moakconstant}
M_q:=\log(\sqrt{1-q})+\log(q;q)_\infty-\frac{\pi^2}{6\log q}.
\end{equation}
The formula converges to the classical Stirling's formula as $q\to 1^-$. In analogy with the classical case, we define
\begin{equation}\label{muq}
\mu_q(x):=\log \Gamma_q(x)- (x-1/2)\log\left(\frac{1-q^x}{1-q}\right)-\frac{1}{\log q} \int_0^{-x\log q} \frac{u}{e^u-1}\, du-M_q.
\end{equation}

\begin{lem}
The function $\mu_q(x)$ is decreasing for $x>0$.
\end{lem}

\begin{proof}
 Differentiating \eqref{muq} we get
\begin{equation}\label{muq'}
\mu_q^\prime(x)=\psi_q(x)-\log\left(\frac{1-q^x}{1-q}\right)-\frac{\log q}{2}\frac{q^x}{1-q^x}.
\end{equation}
where, by \eqref{psi defn}, the digamma function $\psi_q$ is given by
\[
\psi_q(x)=-\log(1-q)+\log q\sum_{j=0}^\infty \frac{q^{x+j}}{1-q^{x+j}}.
\]
One case of the Euler-Maclaurin summation formula (see \cite{aar} for more details) states that if $f$ is a differentiable function with fast enough decay at $\infty$ then we have
\[
\sum_{j=0}^\infty f(j)=\frac{ f(0)}{2}+\int_0^\infty f(y)dy+\int_0^\infty\left(y-[y]-\frac{1}{2}\right)f^\prime (y)\, dy.
\]
Let $x>0$ be arbitrary but fixed. We apply this formula to $f(y)=f_x(y)=\frac{q^{x+y}}{1-q^{x+y}}$ to get
\[
\sum_{j=0}^\infty \frac{q^{x+j}}{1-q^{x+j}}=\frac 12 \frac {q^x}{1-q^x}+\int_0^\infty \frac{q^{x+y}}{1-q^{x+y}}dy+\int_0^\infty\left(y-[y]-\frac 12\right) \log q \frac{q^{x+y}}{(1-q^{x+y})^2}dy.
\]
Hence
\begin{equation}\label{psiapprox}
\begin{split}
\mu^\prime_q(x)&=\psi_q(x)+\log(1-q)-\frac{\log q} 2 \frac{q^x}{1-q^x}-\log(1-q^x)\\
&=\log^2q \sum_{n=0}^\infty \int_n^{n+1} \left(y-n-\frac 12\right)\frac{q^{x+y}}{(1-q^{x+y})^2}dy.
\end{split}
\end{equation}
We shall show that each integral in the last sum is negative. Consider the change of variable $z=2n+1-y$, which transforms the interval $[n, n+\frac 12]$ into $[n+\frac 12, n+1]$. For such $y$ and $z$ we have $y-n-\frac 12=-(z-n-\frac 12)<0$ and
\[
\frac{q^{x+y}}{(1-q^{x+y})^2}>\frac{q^{x+z}}{(1-q^{x+z})^2}
\]
It follows that
\[
\int_{y=n}^{n+\frac 12} (y-n-\frac 12)\frac{q^{x+y}}{(1-q^{x+y})^2}dy<-\int_{z=n+\frac 12}^{n+1}(z-n-\frac 12)\frac{q^{x+z}}{(1-q^{x+z})^2}dz
\]
which immediately gives $\int_{n}^{n+1}(y-n-\frac 12) \frac{q^{x+y}}{(1-q^{x+y})^2} dy <0$, and the result follows.
\end{proof}

We set
\begin{equation}
g_q(x):=\mu_q(x)-\mu_q(x+1)=(x+\frac 12) \log\left(\frac{1-q^x}{1-q^{x+1}}\right)+\frac {1} {\log q} \int_{-x\log q}^{-(x+1)\log q} \frac{u}{e^u-1}\, du.
\end{equation}
Since
\[
\sum_{k=0}^{n-1}g_q(x+k)=\mu_q(x)-\mu_q(x+n),
\]
and $\lim_{x \to \infty} \mu_q(x)=0$, we see that the infinite series
$
\sum_{k=0}^\infty g_q(x+k)
$
converges.

Next, we give the following functional characterization of $\mu_q$ (and hence of the $q$-Stirling formula).

\begin{thm}\label{qstirchar}
Let $\nu(x;q)$ be a function defined for $x>0$ and $0<q<1$ satisfying the following properties
\begin{itemize}
\item[(i)] For fixed $q$, $\nu(x;q)$ is decreasing in $x$.
\item[(ii)] $\nu(x;q)-\nu(x+1;q)=g_q(x).$
\item[(iii)] $\nu(1;q)=\sum_{n=0}^\infty g_q(1+n).$
\end{itemize}
Then for all $x$ and $q$, $\nu(x;q)=\sum_{n=0}^\infty g_q(x+n).$
\end{thm}

\begin{proof}
Note that $\nu(x;q)-\nu(x+1;q)=g_q(x)$ implies that
\begin{equation}\label{nu}
\nu(x+n;q)=\nu(x;q)-\sum_{k=0}^n g_q(x+k).
\end{equation}
For $x \in (0,1]$ we have $\nu(n;q)\geq \nu(n+x;q)\geq \nu(n+1)$ since $\nu$ is decreasing. Applying \eqref{nu} we get
\[
\nu(1;q)-\sum_{k=1}^{n-1}g_q(k) \geq \nu(x;q)-\sum_{k=0}^n g_q(x+k)\geq \nu(1;q)-\sum_{k=1}^{n}g_q(k).
\]
Thus
\begin{equation}\label{ineq}
0\leq  \nu(x;q)-\sum_{k=0}^n g_q(x+k)- \left(\nu(1;q)-\sum_{k=1}^{n-1}g_q(k)\right) \leq g_q(n).
\end{equation}
But $\lim_{n\to \infty}g_q(n)=0$ since $\sum_{k=0}^\infty g_q(1+k)$ is convergent, and the result follows by taking the limit of \eqref{ineq} and using property (iii).
\end{proof}

\begin{cor}\label{stirlingcor}
Let $\phi:(0,\infty)\times(0,1)\to (0,\infty)$ be a function satisfying
\begin{itemize}
\item[(i)] $\phi(x+1;q)=[x]_q\phi(x;q)$,
\item[(ii)] For large $x$, $\phi$ has the asymptotic expansion
\[
\log \phi(x;q)\sim \left(x-\frac12\right)\log\left(\frac{1-q^x}{1-q}\right)+\frac{1}{\log q} \int_0^{-x\log q} \frac{u}{e^u-1}\, du,
\]
\item[(iii)]
$\nu(\phi;x;q):=\log \phi(x;q)-\left[\left(x-\frac12\right)\log\left(\frac{1-q^x}{1-q}\right)+\frac{1}{\log q} \int_0^{-x\log q} \frac{u}{e^u-1}\, du\right]$ is decreasing in $x$,
\item[(iv)] $\phi(1;q)=1,$
\end{itemize}
then $\phi(x;q)=\Gamma_q(x)$.
\end{cor}
\begin{proof}
Conditions (i)-(iii) immediately give that $\nu(\phi;q;x)$ satisfies the first two conditions of Theorem \ref{qstirchar}, while the fourth condition guarantees $\mu_q(1)=\nu(\phi;q;1)$. It follows that for $x>0$ $\mu_q(x)=\nu(\phi;q;x)$ and consequently  $\phi(q;x)=\Gamma_q(x)$.
\end{proof}

\begin{rem}
Using \eqref{gammaq-1} and \eqref{qStirling}, we see that for $q>1$ and large $x$ we have
\be\label{q1stirling}
\Gamma_q(x)\sim \frac{1-x^2}{2} \log q +\left(x-\frac12\right)\log\left(\frac{q^x-1}{q-1}\right)-\frac{1}{\log q} \int_0^{x\log q} \frac{u}{e^u-1}\, du-M_{q^{-1}},
\ee
where $M_{q^{-1}}$ is given by \eqref{moakconstant}. We also see that if $f:(0,\infty)\times (1,\infty)\to  (0,\infty)$ is a function for which $q^{\frac{(x-1)(x-2)}{2}}f(x;q^{-1})$ satisfies the conditions of Corollary \ref{stirlingcor}, then we must have $f(x;q)=\Gamma_q(x)$. It follows that \eqref{char2} and \eqref{qStirling} (resp. \eqref{q1stirling}) uniquely define $\Gamma_q(x)$ for $q\in(0,1)$ (resp. $q>1$).
\end{rem}
\noindent{\bf Acknowledgment}
It is the authors' pleasure  to thank Professor Mourad Ismail for suggesting the idea of this paper and providing them with reference \cite{anasta}, and for informing them about \cite{Gosper2001}. We also express our gratitude to the anonymous referee for constructive suggestions that connected our results to those in \cite{Gosper2001} and \cite{Mezo2012}.

\bibliographystyle{plain}

\end{document}